\documentclass{article}

\usepackage{lineno,hyperref,cite}
\usepackage{graphicx,amsmath,caption,epstopdf,authblk}
\usepackage{mathrsfs}
\usepackage[utf8]{inputenc}
\usepackage[english]{babel}
\usepackage{amsthm,amsmath}
\usepackage{caption}
\usepackage{subcaption}
\usepackage[letterpaper, margin=1in]{geometry}

\newtheorem{theorem}{Theorem}

\newtheorem{lemma}{Lemma}

\title{Application of neural-network hybrid models in estimating the infection functions of nonlinear epidemic models}
\author[a]{Chentong Li}
\author[b]{Changsheng Zhou}
\author[c]{Junmin Liu}
\author[d]{Yao Rong}
\affil[a]{Guangdong Key Laboratory of Modern Control Technology, Institute of Intelligent Manufacturing, 
Guangdong Academy of Science, Guangzhou, Guangdong 510070, P.R.China}
\affil[b]{School of Mathematics and Information science, Guangzhou University,
 Guangzhou, Guangdong 510006, P.R.China}
\affil[c]{School of Mathematics and Statistics, Xi'an Jiaotong University,
 Xi'an, Shaanxi 710049, P.R.China}
\affil[d]{College of Engineering Physics, Shenzhen Technology University, Shenzhen, Guangdong  518118, P.R.China }
\date{}

\begin{document}

\maketitle

\begin{abstract}
Hybrid neural-network models combine the advantages of a neural network’s fitting functionality with differential equation models to reflect actual physical processes and are widely used in analyzing time-series data. Most related studies have focused on linear hybrid models, but only a few have examined nonlinear problems. In this work, we use a hybrid nonlinear epidemic neural network as the entry point to study its power in predicting the correct infection function of an epidemic model. To achieve this goal, we combine the bifurcation theory of the nonlinear differential model with the mean-squared error loss and design a novel loss function to ensure model trainability. Furthermore, we find the unique existence conditions supporting ordinary differential equations to estimate the correct infection function. Using the Runge--Kutta method, we perform numerical experiments on our proposed model and verify its soundness. We also apply it to real COVID-19 data to accurately discover the change law of its infectivity. 

{{\bf Key Words:} Differential equations; Epidemic model; Hybrid model; Neural network. }
\end{abstract}

\section{Introduction}

Forecasting an epidemic’s infection number and rate is a key objective in the study of infectious diseases. The ordinary differential equation model of infectious disease is based on known disease properties and has long been used to analyze the dynamic behaviors of infection and to provide forecasts \cite{du2017evolution, anderson1992infectious, hufnagel2004forecast}. All infectious disease models can be divided into two parts: infection terminology and susceptibility  \cite{anderson1992infectious, ma2009modeling}. The infected terms are used to illustrate the change laws of the number of infected persons at different periods (e.g., incubation, infection, and isolation). Most epidemic models use nonlinear infection functions. However, nonlinearity quickly spawns from the many combinations of infection and susceptibility terms \cite{ma2009modeling}. To deal with the difficulties of nonlinearity, researchers usually use the bifurcation theory of nonlinear equations to make estimations; then, they improve upon them with expert knowledge \cite{ma2009modeling}. 

To implement nonlinear epidemic models, the basic reproduction bifurcation number, $R_0$, comprises the parameters and functions of the model \cite{wang2008threshold}. This number is used to describe the expected number of infected cases based on one original infected person \cite{wang2012basic}. Used as a biological descriptor, $R_0$ is the threshold value used to determine whether the disease may vanish. In many epidemic models, when $R_0$ is larger than $1$, the epidemic will persist forever; if it is smaller than $1$, the number of infected individuals will approach $0$ over time \cite{wang2012basic}. This persistence description is called ``forward bifurcation’’ \cite{martcheva2019methods} and is the key output of epidemic modeling.

Although nonlinear epidemic model analysis has come a long way, selecting the correct infection function (also known as the infection incidence) remains prohibitively difficult \cite{zhang2007periodic,smith1983multiple,alexander2004periodicity, cui2008impact, ruan2003dynamical, xiao2007global}. The functions not only pertain to the number of susceptible and infected individuals, but they also involve environmental and human behavior factors. For example, the periodical temperature changes in temperate zones may influence the activity of viral proteins, causing infectability to reflect a periodical property \cite{zhang2007periodic, smith1983multiple}. Additionally, when the infected number grows, the growth rate of infectability may become smaller, depending on human behaviors. Mathematically, the Holling type function is applied to the infection function to account for such behavior. However, the many combinations of factors create considerable uncertainty \cite{ruan2003dynamical,xiao2007global}. Thus, any inaccuracy in the choice of infection function directly influences the precision of model prediction .

Some statistical methods have been applied to real epidemic data using Akaike's information criterion (AIC) \cite{yamaoka1978application, yamashita2007stepwise} or its second-order estimate (AICC) \cite{hurvich1989regression}. Using these, one can ostensibly find the most suitable function from a group of candidates. In many cases, however, relying on the available data is insufficient, and expert knowledge of the environment, human behavior, and the properties of the viral strain is also needed. Therefore, better methods are necessary.

Notably, improvements to deep learning \cite{lecun2015deep} and related neural-network models have greatly enhanced the predictability of nonlinear functions throughout modern industry. A neural network is modeled after the neuromorphology of the human brain using network to construct input layers, output layers, neurons, and activation functions (synapses) at each layer. As with many physical functions, it sacrifices explainability for precision. Hence, at its heart is a black-box function that is trained to ``think’’ in different dimensions \cite{hill1994artificial}.  According to the universal approximation theory \cite{hornik1989multilayer}, a neural network can approximate any functions given the ssoftwareuitable activation functions. Therefore, if we can build a neural network to implement the differential equations of an infectious disease model, infection model prediction should become much more accurate. 

As introduced by Psichogios and Ungar \cite{psichogios1992hybrid}, combining capabilities in this way results in a ``hybrid’’ neural-network model. The authors built a hybrid model to forecast biological phenomena. Later, \cite{guo2001simulation, lee1996hybrid} used them to predict the change laws of time-series data. Recently, the authors of \cite{ren2019hybrid, razmjooy2018hybrid, wu2020drug} had breakthrough success in forecasting biological processes. Notably, biological systems contain both first-principle information (i.e., biophysical processes) and unknown information, which accounts for the complexity of a biophysical system. Hence, the problem now reduces to properly fitting the data to the biophysical process.

Most hybrid neural-network studies have leveraged the linear-first principle to effectively side-step overly complex nonlinear preparation methods. Thus, questions have been raised about whether the hybrid functions can be adequately trained with limited data and whether the estimates are actually reliable. In this work, based on the standard theorem of forwarding bifurcation, a loss function is created that overcomes the untrainability of the nonlinear epidemic model. Furthermore, based on the unique conditions of ordinary differential equations and the basic properties of epidemic models, we provide the conditions needed to predict the correct infection function.

The remainder of this work is organized as follows. In the second section, we provide complete descriptions of the nonlinear epidemic system and introduce the basic conditions of ordinary differential epidemic modeling. In the third section, the two theorems that can solve the trainable and estimable problems of the neural network are explained. In section four, based on the standard Runge--Kutta method of numerically solving ordinary differential equations, we introduce a method of properly training the hybrid neural network. In the fifth section, the results of four numerical experiments are reported to demonstrate the efficacy of the proposed model in estimating periodical and Holling-type infection functions. Then, in the sixth section, we report on the application of the proposed method to real data. In the last section, we conclude this work and discuss future opportunities.

\section{Hybrid epidemic model properties and conditions}

The hybrid nonlinear epidemic neural network model is written as follows:
\begin{equation}
    x'=g(x,f),
    \label{eq:1}
\end{equation}
where $x=(x_S,x_I)$ and $x_S\in R^{m_1}$ are the variables of a susceptible person, and $x_I\in R^{m_2}$ are those of an infected person. Function $f$ maps time $t$ or variables $x$ into $R$. Function $g=(g_S,g_I)$, where $g_S$ is the function mapping of $x_S$ and the output of function $f$ from $R^{m_1+1}$ to $R^{m_1}$. Additionally, $g_I$ is the function mapping of $x_I$ and the output of function $f$ from $R^{m_2+1}$ to $R^{m_2}$. Furthermore, we assume that function $g$ satisfies the following three conditions:

(C1) $\forall x_1\neq x_2$, there exist a positive constant, $L_1$, such that $\|g(x_1,y)-g(x_2,y)\|<L_1\|x_1-x_2\|$; 

(C2) $\forall y_1\neq y_2$, there exist a positive constant, $L_2$, such that $\|g(x,y_1)-g(x,y_2)\|<L_2\|y_1-y_2\|$;

(C3) if $x_I \neq 0$, then $\forall y_1\neq y_2$, and the in-equation $\|g_I(x,y_1)-g_I(x,y_2)\| > 0$ hold.

Condition (C1) is the famous Lipschitz condition of the ordinary differential equation and is sufficient for the equation’s unique existence \cite{walter1998ordinary}. Condition (C2) is the Lipschitz condition for the function, which, when used as the parameter, accommodates the unique existence of its corresponding ordinary differential equation. In the proofs shown in the next section, we replace $L_1$ and $L_2$ in (C1) and (C2), respectively, with $L=\max(L_1,L_2)$. Condition (C3) guarantees that, when $x_I \neq 0$, map $g$ is the injective map for $y$. These three conditions are used in the upcoming proofs.

Let $x(x_0,f,t)=(x_S(x_0,f,t),x_I(x_0,f,t))$ be the solution to Equation \eqref{eq:1} with respect to function $f$ and the initial condition, $x_0$. Then, based on the standard of nonlinear epidemic models \cite{ma2009modeling}, we can assume that two more conditions are satisfied by the system of Equation \eqref{eq:1}.

Condition (C4), without $x_I$ in system $x_S'=g_S(x_S,f)$, has the unique positive equilibrium, $S^*$.

Condition (C5) allows vector $0\in R^{m_2}$ to become the zero vector, and when $x_I=0$, $g_I(x,f)=0$.

Conditions (C4) and (C5) are built upon the properties of infectious diseases. If no disease exists, the system only contains susceptible individuals, and the system has a unique positive state. This, the without-disease system will remain in a stable no-disease state until a new infected person enters the system. Therefore, the system of Equation \eqref{eq:1} has a unique disease-free equilibrium, $x_{DFE}=(S^*,0)$.

\begin{figure}[htbp]
  \centering
  \subcaptionbox{Forward bifurcation diagram.}{\includegraphics[width=0.48\textwidth]{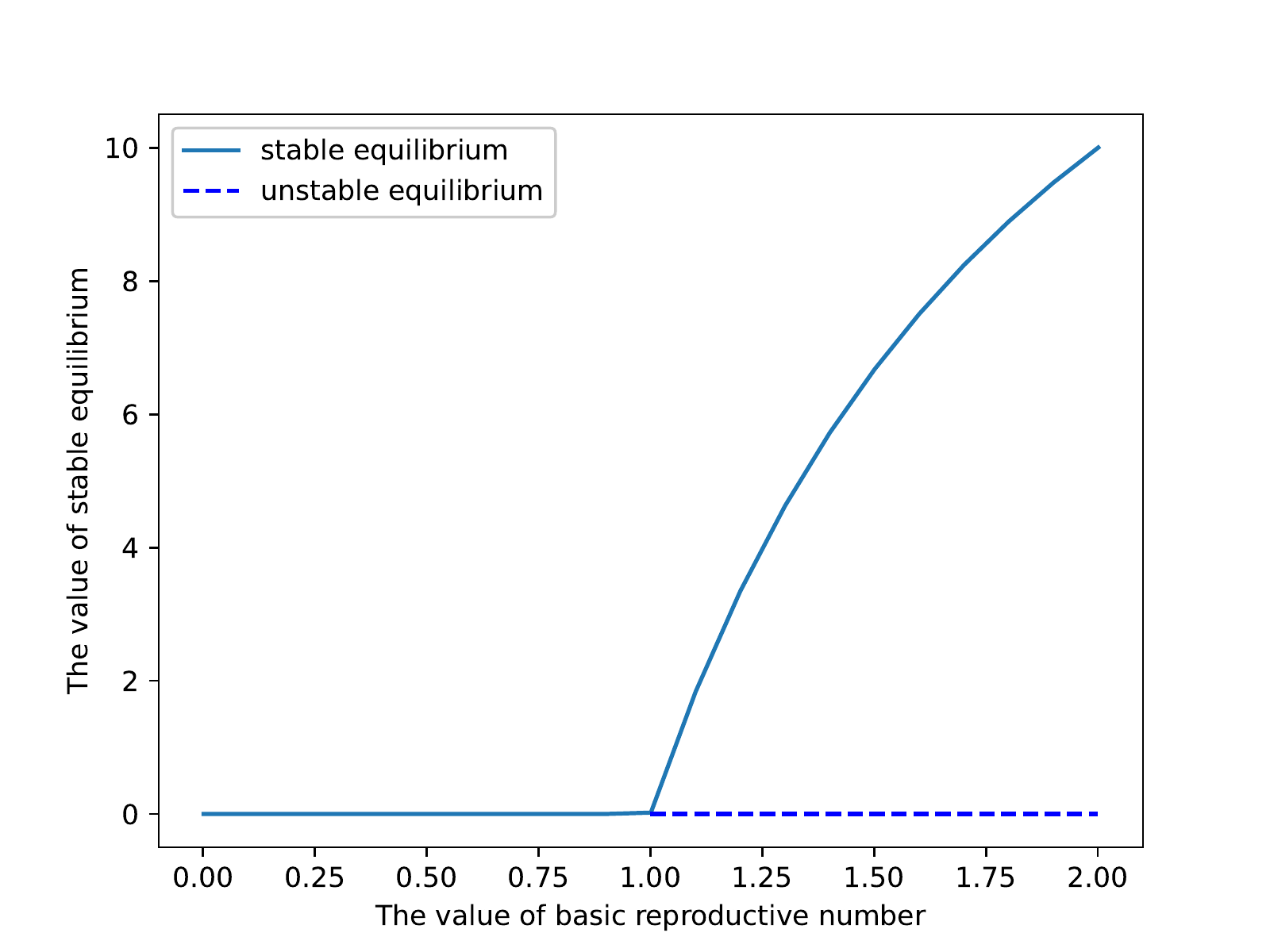}}
  \hfill
  \subcaptionbox{Solutions under different conditions.}{\includegraphics[width=0.48\textwidth]{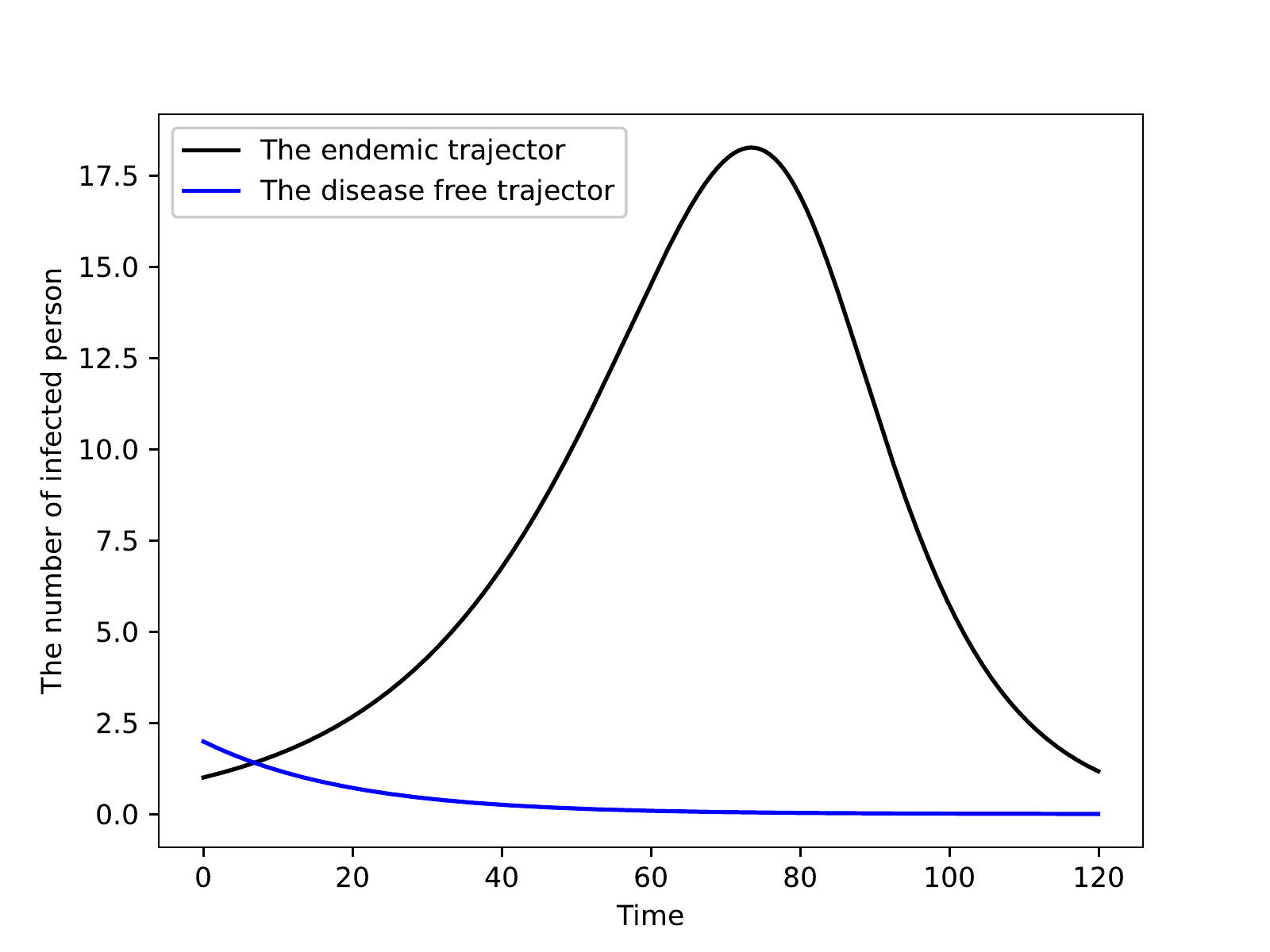}}
  \caption{Forward bifurcation and solutions under different bifurcation conditions.}
  \label{fig:R0}
\end{figure}

Based on the standard forward bifurcation theory of the nonlinear epidemic model \cite{martcheva2019methods}, there exists a threshold value, $R_0>0$, the basic reproductive number, which determines whether a disease outbreak occurs. This value applies to most epidemic model s. Therefore, we assume that the system of Equation \eqref{eq:1} satisfies the following forward bifurcation condition:

In the (C6) condition, there exists a function, $R_0^f$, that maps the parameters and function $f$ of the system of Equation \eqref{eq:1} to a positive real number, $R_0$. This map is continuous, and when $R_0<1$, the disease-free equilibrium, $x_{DFE}$, of the ordinary differential equation system, Equation \eqref{eq:1}, is globally stable; when $R_0>1$, the disease-free equilibrium is broken.

Condition (C6) represents the standard result of most epidemic systems. Hence, research on nonlinear epidemic models seek to find the basic reproductive number \cite{zhang2007periodic,wang2008threshold,wang2012basic}. Figure \ref{fig:R0} illustrates a simple forward bifurcation epidemic system. Figure \ref{fig:R0}(a) shows that when $R_0$ is smaller than $1$, the disease-free equilibrium is stable, and when this number is larger than $1$, the disease-free equilibrium becomes unstable. Meanwhile, Figure \ref{fig:R0}(b) shows the solution to $x_I$ under different bifurcation conditions. When $R_0>1$, the solution appears as the black line in Figure \ref{fig:R0}(b), and when $R_0<1$, the solution appears as the blue line in Figure \ref{fig:R0}(b). 
The above six conditions are the basic properties satisfied by most epidemic models. 

\section{Hybrid model trainability and estimability}

To finish the proofs, we first introduce some more notations. Let the neural-network model, $f_{\theta}(x) = \phi_n\circ\cdots\circ\phi_2\circ\phi_1(x)$, where $\phi_i(x)=\sigma(w_ix+b_i)$ is the linear combination of the matrix, $w_i$, and the bias vector, $b_i$, with the activation function, $\sigma$. The constant, $n>0$, is the number of layers in the neural network, $f_{\theta}$, and $\theta$ is the parameter set, including the weights in $w_i$ and $b_i$. All activation functions used in this study consist of elemental functions. We also note the norm, $\|(x_1(t),x_2(t),\cdots,x_m(t))\|=\sum_{i=1}^m\int_0^Tx_i(t)^2dt$, where $(0,T)$ is the collected area of the data. Then, the following Lemma can be directly obtained:

\begin{lemma}
The neural-network function, $f_{\theta}:R^n\to R^m$, is continuous in $x\in(0,T)^n$ with the bounded weight parameter, $\theta=(w,b)$, and elemental activation functions $\sigma$. 
\label{lemma:fc}
\end{lemma}
This can be proven directly by using the continuities of the elemental functions.

When estimating function $f$, the most common data type is the number of infected individuals. Thus, we use the distance between the terms of infected individuals of Equation \eqref{eq:1} and the data of the infected individuals as the loss function.  The trivial formula of the loss can be written as $L_{\theta}=\|x_I(x_0,f_{\theta},t) - D(t)\|$, where $D(t)$ are the data. However, this formula is not suitable for Equation \eqref{eq:1}. Based on the following theorem, we prove that $f_{\theta}$ cannot be trained from data $D(t)$.

\begin{theorem}[Vanishing gradient]
The system of Equation \eqref{eq:1} satisfies conditions (C1)--(C6) if there exists a parameter set, $A\neq \emptyset$, which is a subset of $R^n$, and $\forall\theta\in A$, the 
basic reproductive number satisfies $R^f_0(f_{\theta})<1$. Then, $\forall\epsilon>0$ and $D(t)$, there exists a positive constant, $\delta$, positive initial condition $x_0$, and parameters $\theta$, such that the gradient of the loss function, $L_{\theta}=\|x_I(x_0,f_{\theta},t) - D(t)\|$, with respect to the neural-network parameters, satisfies 
$\nabla_{\theta} L_{\theta}< \epsilon$.
\label{th:gradient}
\end{theorem}
\begin{proof}
When $\theta\in A$, by (C6), the disease-free equilibrium, $x_{DFE}=(S^*,0)$, is stable. Thus it follows that $\forall\epsilon>0$, there exists a constant, $\delta>0$, for all initial conditions, $\|x_0-x_{DFE}\|<\delta$, such that, for every $t\ge 0$, $\|x_I(x_0,f_{\theta},t)-0\|=\|x_I(x_0,f_{\theta},t) \|<\epsilon/2$. Meanwhile, for another parameter $\hat{\theta}\neq\theta$ and $\hat{\theta}\in A$, we could find a $\hat{\delta}>0$, such that for all
 initial conditions $\|x_0-x_{DFE}\|<\hat{\delta}$, we have $\|x_I(x_0,f_{\hat{\theta}},t)\|<\epsilon/2$. Then, we let $\delta'=\min(\delta,\hat{\delta})$, and $\|x_0-x_{DFE}\|<\delta'$. Thus, we obtain the following equations:
 
$$
\begin{aligned}
   L_{\theta} - L_{\hat{\theta}}  & \le \|x_I(x_0,f_{\theta},t) -  D(t)\| - \|x_I(x_0,f_{\hat{\theta}},t) - D(t)\| \\,
   & \le (\| D(t) \| + \|x_I(x_0,f_{\theta},t) \|) - (\| D(t) \| -\|x_I(x_0,f_{\hat{\theta}},t) \|) \\,
   & \le (\| D(t) \| + \epsilon/2) - (\| D(t) \| -\epsilon/2) \\
   & = \epsilon.
\end{aligned}
$$

By the continuity of $f_{\theta}$ (Lemma \ref{lemma:fc}) and $R_0(x)$, there exists a positive constant, $a$, such that the small ball, $B(\theta,a)\in A$. Therefore, by the above equations, $\forall \hat{\theta}\in B(\theta,a)$, we have $ |L_{\theta} - L_{\hat{\theta}}| < \epsilon$, which follows $\nabla_{\theta} L_{\theta}< \epsilon$. 
\end{proof}

The blue line of Figure \ref{fig:R0}(b) shows one special solution of the epidemic model when $R_0<1$, in which the solution approaches zero very fast. The same is true for the other solution  of the same model, making the difference between them negligible. This causes the gradient of the loss function to become rather small, which is intuitive.

In the real world, almost all detected diseases have basic reproductive numbers that are larger than one. However, in the training process, the trained function, $f_{\theta}$, may appear in an area that makes the number smaller than one. Therefore, based on the above theorem, the hybrid model is untrainable.
For this reason, we introduce a loss function that combines the standard loss with the basic reproductive number, $R_0$, which is the bifurcation parameter of the nonlinear epidemic system, Equation \eqref{eq:1}. This makes the loss function non-zero when $R_0$ is smaller than $1$.
\begin{equation}
     L_{\theta} = \|x_I(x_0,f_{\theta},t) - D(t)\| + \alpha \max(1 - R^f_0(f_{\theta}), 0), 
     \label{eq:m}
\end{equation}
where $\alpha>0$ is the hyper-parameter determined by prior knowledge. In this loss function, when $R_0<1$, according to Theorem \ref{th:gradient}, loss $\|x_I(x_0,f_{\theta},t) - D(t)\|$ may be very small. However, when the value of $\max(1 - R^f_0(f_{\theta}), 0)$ is larger than zero, and its minimal value is zero, the gradient of that part, hence the entire loss function, becomes non-zero. When $R_0>1$, the loss function becomes standard again. For other hybrid nonlinear models with different bifurcation parameters, we can construct suitable bifurcation items in the loss function to ensure trainability. 

From the above analysis, we have obtained a suitable loss function that allows the hybrid nonlinear neural-network model to be trained. However, the question remains of whether the infection function can be learned from real data. Therefore, we introduce the following loss function:
\begin{equation}
    L_{\theta} = \|x_I(x_0,f_{\theta},t) - x_I(x_0,f^*,t)\| + \alpha \max(1 - R^f_0(f_{\theta}), 0). 
    \label{eq:min}
\end{equation}
We must find whether the real-world infection function, $f^*$, can be fitted by neural-network $f_{\theta}$. Using the following theorem, we prove that the system of Equation \eqref{eq:1} fits function $f^*$ using the loss function of Equation \eqref{eq:min}.

\begin{theorem}[Existence] 
The system of Equation \eqref{eq:1} satisfies Conditions (C1)--(C6) if $R^f_0(f^*)>1$. Thus, $\forall \epsilon>0$, $f_{\theta}$ satisfies $\|f_{\theta} - f^*\|<\epsilon$ if and only if there exists a parameter set, $\theta$, such that the loss function, $L_{\theta}$, in Equation \eqref{eq:min} satisfies $L_{\theta}<\epsilon$.
\label{th:existence}
\end{theorem}
\begin{proof}
The solution to the system of Equation \eqref{eq:1} with an initial time of zero and $x_0$ as the initial condition is written as 
$x(x_0,f_{\theta},\Delta t)=x_0+\int_0^{\Delta t}g(x,f_{\theta})dt$.
Therefore, we let $x^{\theta}_{1}=x_0 + \int_0^{\Delta t}g(x,f_{\theta})dt$, 
$x^*_{1}=x_0 + \int_0^{\Delta t}g(x,f^*)dt$, and $x^{\theta}_{n+1}=x^{\theta}_{n}+\int_{n\Delta t}^{(n+1)\Delta t}g(x,f_{\theta})dt$,
$x^{*}_{n+1}=x^{*}_{n}+\int_{n\Delta t}^{(n+1)\Delta t}g(x,f^*)dt$ for $n=1,2,..,N$,
where $N=\lfloor T/\Delta t\rfloor$. 

By the definition of $\|\cdot\|$, function $\|x(x_0,f_{\theta},t)-x(x_0,f^*,t) \|$ is monotone increasing with respect to $t$. Then, by Conditions (C1) and (C2), the following system of equations holds:

$$
\begin{aligned}
    \|x^{\theta}_{1}-x^{*}_{1} \| &\le \int_0^{\Delta t} \| g(x,f_{\theta})-g(x,f^*) \| dt \\
    & \le \Delta t L(\|f_{\theta} - f^*\| + \|x^{\theta}_{1}-x^{*}_{1} \|) \\
    & \le \Delta t L(\epsilon + \|x^{\theta}_{1}-x^{*}_{1} \|).
\end{aligned}
$$

Let $\Delta t = \frac{1}{2L}$ be the time interval, so that we have $\|x^{\theta}_{1}-x^{*}_{1} \|\le\frac{\Delta t L}{1-\Delta t L}\epsilon = \epsilon$.
Similarly, the following hold:
$\|x^{\theta}_{2}-x^{*}_{2} \|\le \frac{\Delta t L+1}{1-\Delta t L}\epsilon = 3\epsilon$,
and $\|x^{\theta}_{n}-x^{*}_{n} \|\le \frac{\Delta t L+n-1}{1-\Delta t L}\epsilon= (2n-1)\epsilon$.
Thus,

$$
\begin{aligned}
    L_{\theta} &= \|x_I(x_0,f_{\theta},t) – x_I(x_0,f^*,t)\|
               \le \sum_{n=1}^N\|x^{\theta}_{n}-x^{*}_{n} \| \\
               &\le N^2\epsilon \le (2TL+1)^2\epsilon.
\end{aligned}
$$
Let $\epsilon’ = (2TL+1)^{-2}\epsilon$, so that we get $L_{\theta} < \epsilon'$.

On the other hand, supposing there exists a parameter set, $\theta'$, and a positive constant, $\delta$, such that $\|f_{\theta'} - f^*\| > \delta$ and $ L_{\theta'} < \epsilon$, it follows that if $\|f_{\theta}-f^* \|<\epsilon$, then $L_{\theta}<\epsilon'$ by the first part of that proof, and
$\|f_{\theta'} - f_{\theta}\| > \delta-\epsilon>0$.
Thus, there exists a time, $t_c \in (0,T)$, $\forall t \in [0,t_c)$, $f_{\theta'}(t) = f_{\theta}(t)$ and $f_{\theta'}(t_c) \neq f_{\theta}(t_c)$.
Therefore, 

$$
\begin{aligned}
     \|x_I(x_0,f_{\theta},t) - x_I(x_0,f_{\theta'},t) \| & \ge \int_0^{t_c+\delta t} (x_I(x_0,f_{\theta},t) - x_I(x_0,f_{\theta'},t))^2 dt \\
     & = \int_{t_c}^{t_c+\delta t} (\int_{t_c}^{t}g_I(x,f_{\theta}) - g_I(x,f_{\theta'})ds)^2 dt.
\end{aligned}
$$
According to $\|g_I(x,f_{\theta})-g_I(x,f_{\theta'})\|>0$ (Condition (C3)) and the continuity of function $g(x,y)$, there exists a small-enough $\delta t$ that, $\forall t \in (t,t+\delta t)$, $\|g_I(x(t),f_{\theta})-g_I(x(t),f_{\theta'})\|>0$, from which it follows that constant $\delta'>0$ exists, such that $\|g_I(x(t),f_{\theta})-g_I(x(t),f_{\theta'})\|>\delta'$.
Hence, we have $\|x_I(x_0,f_{\theta},t) - x_I(x_0,f_{\theta'},t) \|>\delta'^2\delta t^3/3:=\hat{\delta}>0$.
Thus, $L_{\theta'} > \hat{\delta} - \epsilon'>0$, which leads to a contradiction.
\end{proof}

From the above theorem, if the epidemic hybrid model satisfies the conditions listed in the last section , the form of the infection function in the hybrid model can be estimated using real data. In the next two sections, we explain how to use a suitable numerical method to estimate the infection function.

\section{Numerical methods}
\begin{figure}[htbp]
\centering
  \includegraphics[width=0.9\textwidth]{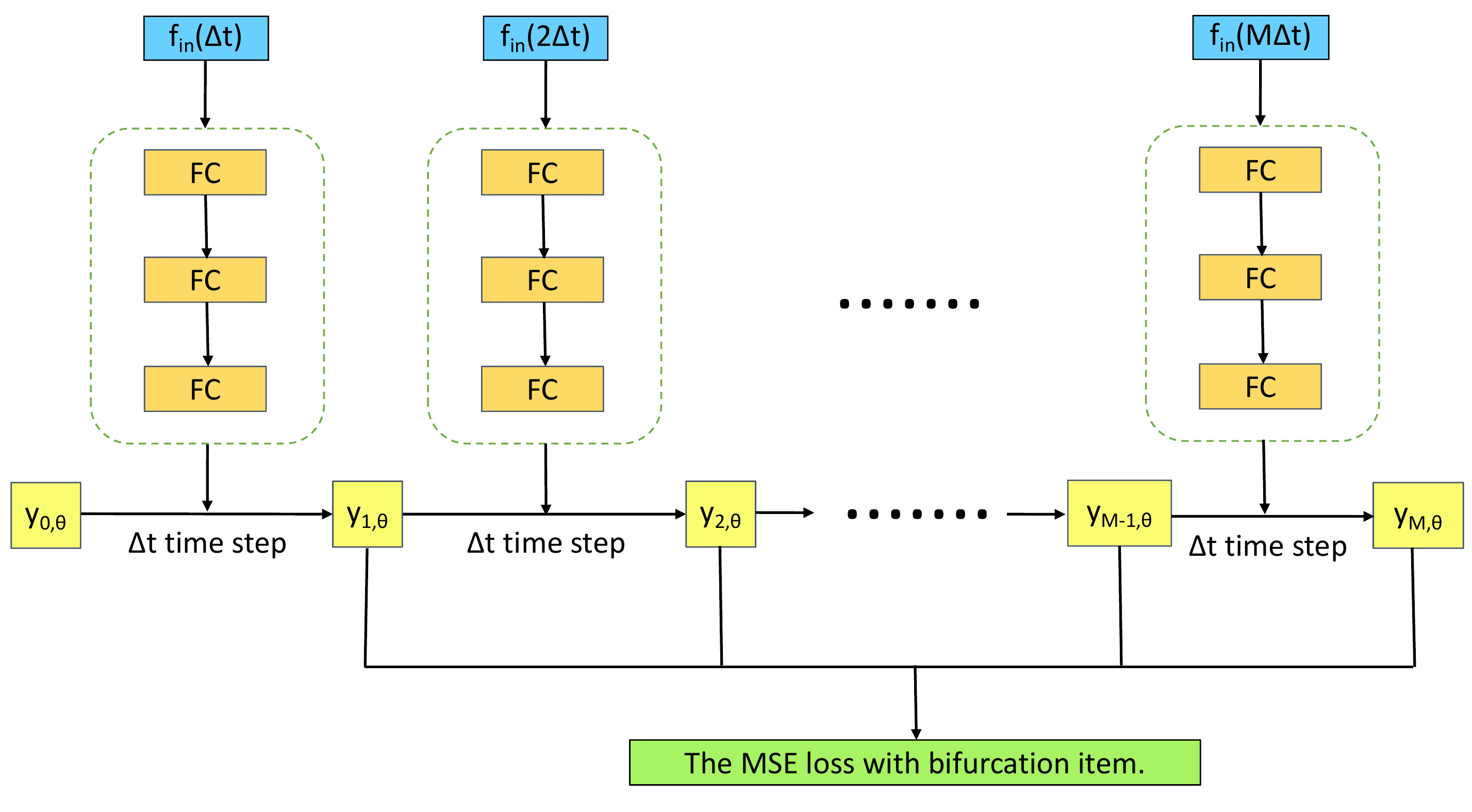}
  \caption{Calculation process.}
\label{fig:algorithm}
\end{figure}
Euler and Runge--Kutta methods are the most effective and most widely used for solving ordinary differential equations \cite{hildebrand1987introduction}. The Runge--Kutta method derives from the integral form of the ordinary differential equation and is used to calculate the numerical solution at each time point. The Euler method derives from the discrete form of differentiation, and it is less accurate. However, it requires fewer computing resources. 

In this work, given a known infection function with small perturbations, the Euler method is used to generate training data. During training, the fourth-order Runge--Kutta method with fixed steps is used to calculate the loss between generated data and the hybrid model. Figure \ref{fig:algorithm} shows the calculation diagram of the numerical method. At each time step of the numerical calculation, the neural-network input is $f_{in}(i\Delta t)$, $i=1,2,\cdots,M$, where the $f_{in}(i\Delta t)=i\Delta t$, or $f_{in}(i\Delta t)=x(i\Delta t)$, and the output is the value of the estimated function at that time. Then, the numerical solution of the hybrid model at each step is used to fit the generated data. In the loss function of Equation \eqref{eq:m}, the norm of $\|\cdot\|$ is difficult to calculate directly. Thus , we use the discrete form of $\|\cdot\|$ (mean-squared error (MSE) loss) to calculate its norm \cite{wackerly2014mathematical}.  By combining the MSE loss and the bifurcation items, we obtain the loss function used for numerical calculation:
\begin{equation}
    L_{\theta} = \sum_{i=1}^{M}(y_{i,\theta}-D_i)^2/M + \alpha\max(1 - R^f_0(f_{\theta}), 0),
    \label{eq:descrete}
\end{equation}
where $M=\lfloor T/dt \rfloor$, $dt$ is the length of the time step, and $T$ is the length of the sample collected area. $y_{i,\theta}=x_I(x_0,f_{\theta},idt), and i=1,2,\cdots,M$ is the numerical solution of the infected term of the system of Equation \eqref{eq:1} at time $idt$. $D_i$ is the collected data at time $idt$. The initial condition applies $x_0$ number of susceptible and infected individuals. The number of infected individuals reflects the initial number of cases, and the number of susceptible persons reflects the entire population minus infected cases. 
\section{Numerical experiments}
In this section, we report the results of four numerical experiments to demonstrate the efficacy of the proposed hybrid neural-network model when applied to estimating periodic and Holling type infection functions. All experiments were performed on an Intel Core i7 -10700 @2.90-GHz computer with 16-GB RAM. All codes were written for the Windows x64 operating system using Pytorch ($>=1.4.0$) and Numpy ($>=1.18.5$) Python packages. Source code can be found at \url{https://github.com/ChentongLi/Inf_Estimation_pytorch}. 

\begin{figure}[htbp]
\centering
\subcaptionbox{Hybrid model fitted result.}{\includegraphics[width=0.48\textwidth]{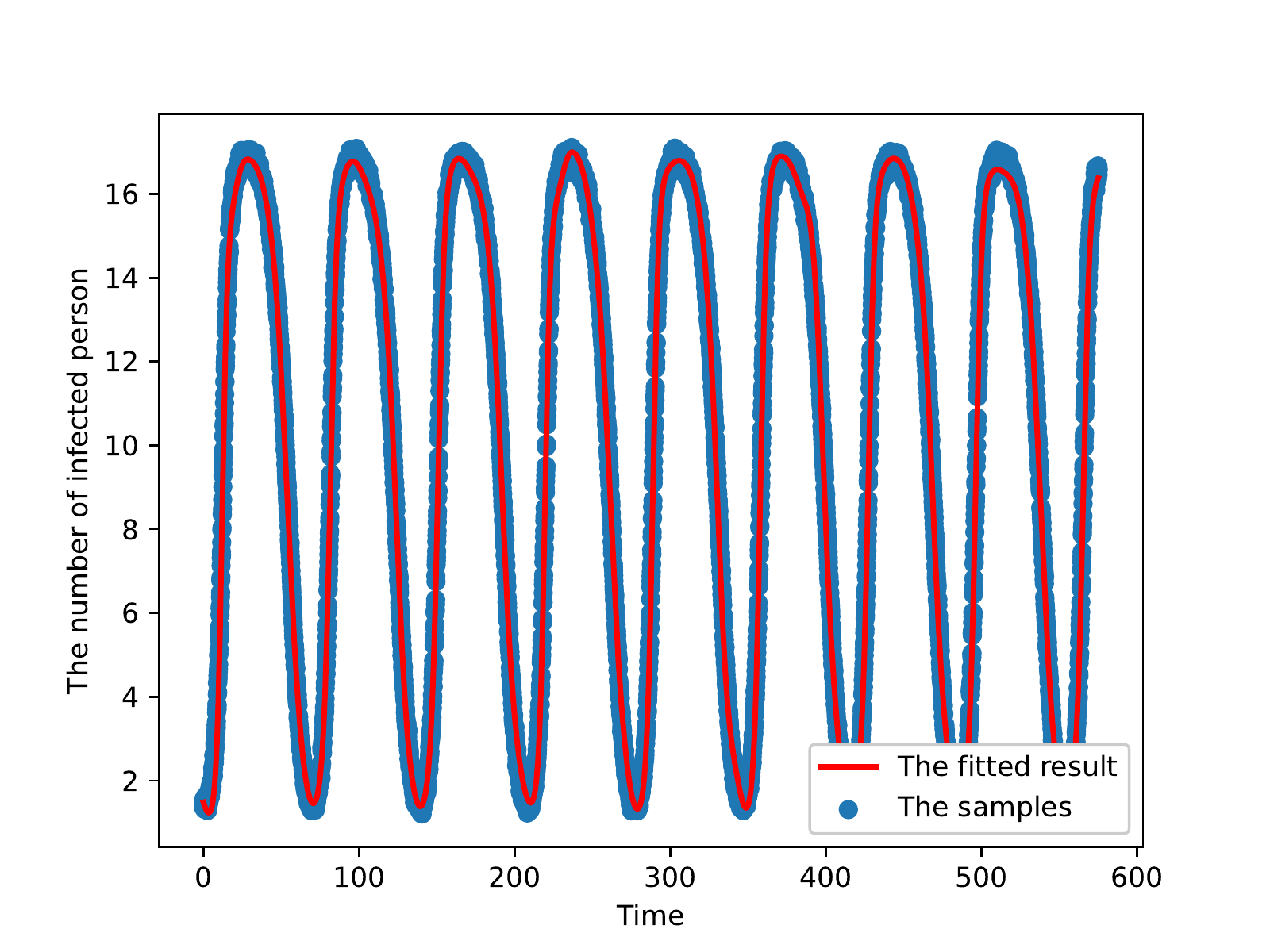}}%
\hfill
\subcaptionbox{Neural network fitted result.}{\includegraphics[width=0.48\textwidth]{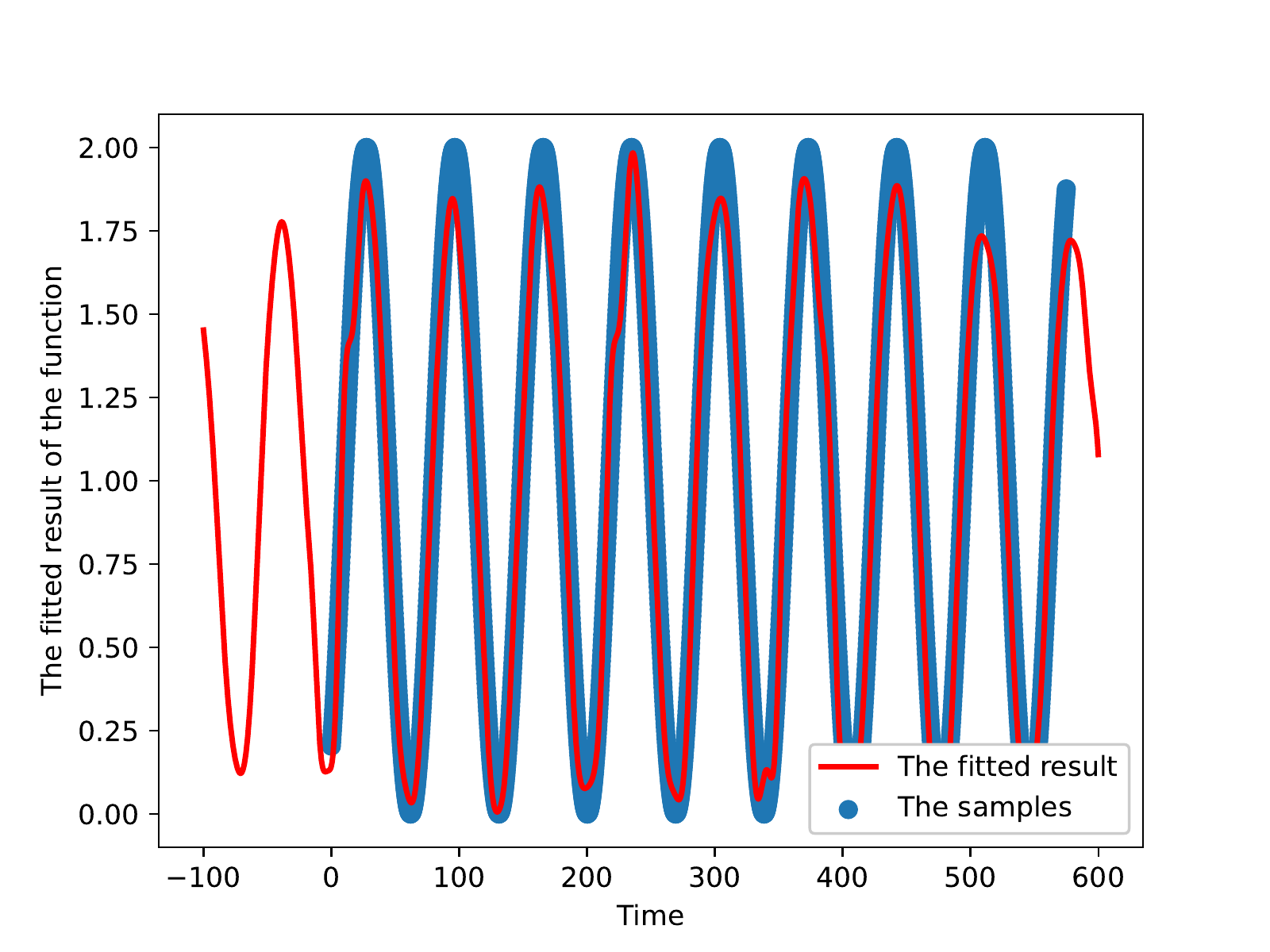}}%

\subcaptionbox{$|\beta(t)-\beta_{\theta}(t)|$.}{\includegraphics[width=0.48\textwidth]{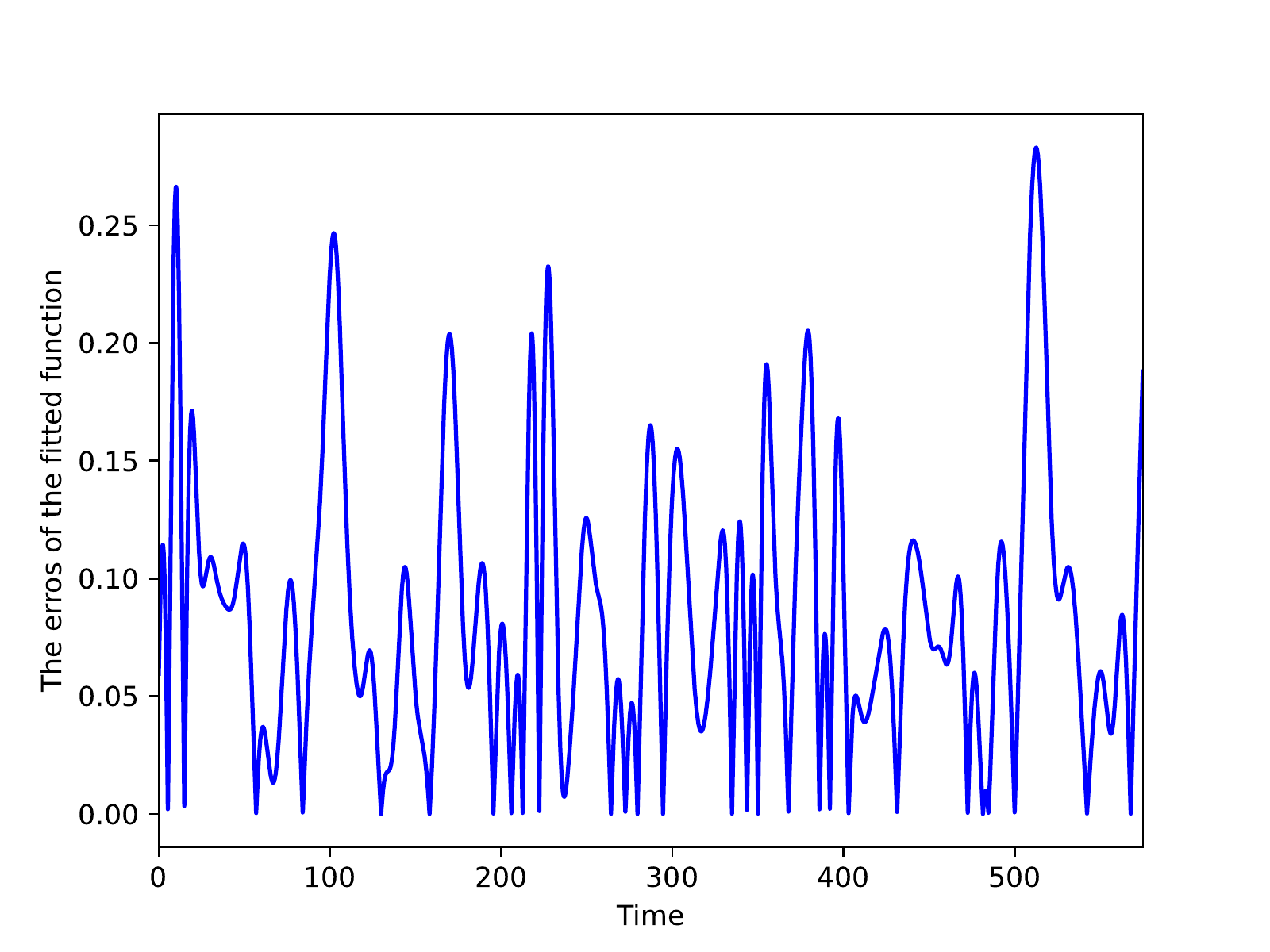}}%
\hfill
\subcaptionbox{Loss with respect to iterations.}{\includegraphics[width=0.48\textwidth]{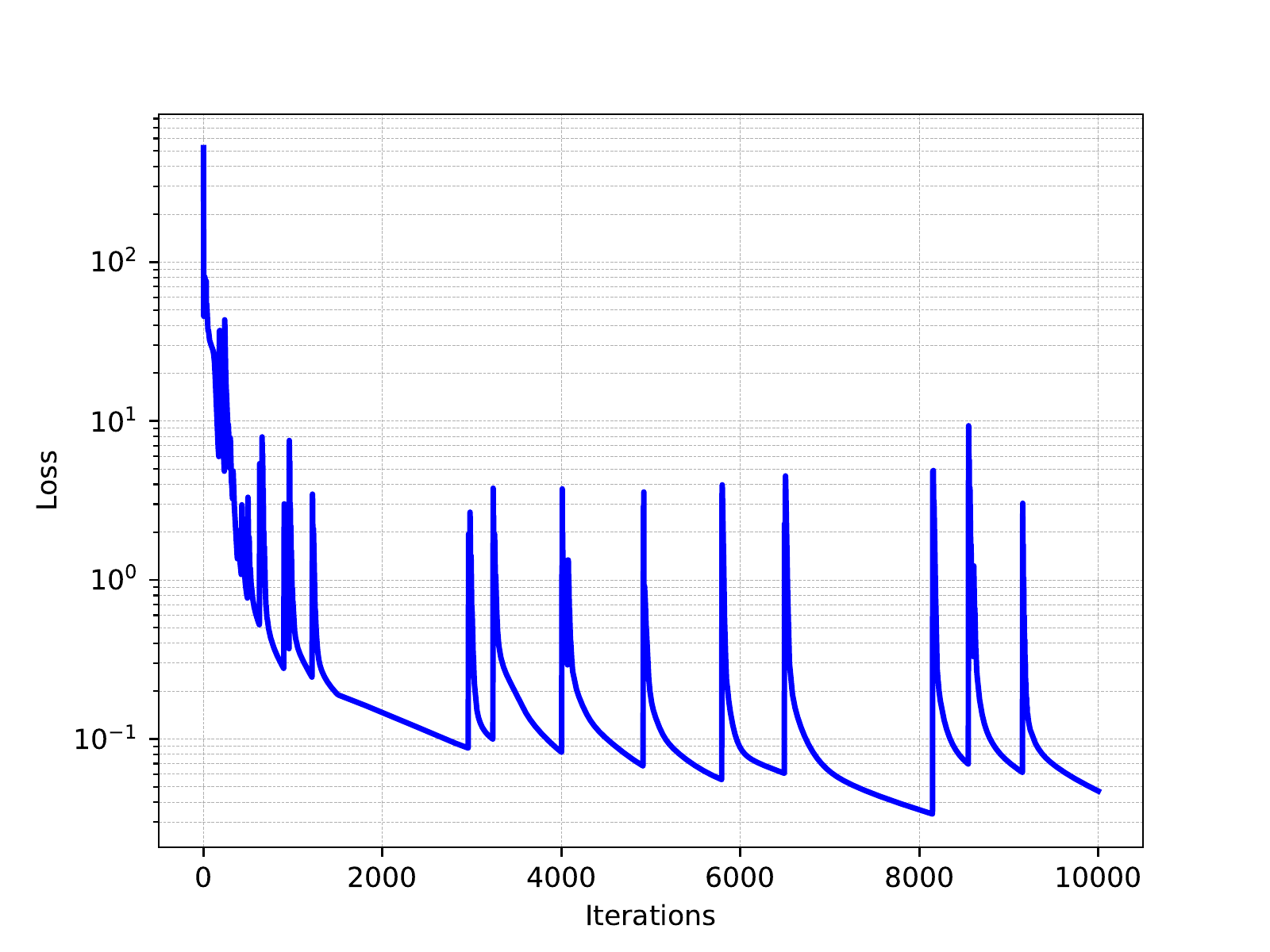}}%
\caption{Fitted results and susceptible--infected--susceptible training process with a simply periodic function: (a) Fitted result of the hybrid model of Equation \eqref{eq:SIS} using generated data; (b) Fitted result of the neural network, $\beta_{\theta}$, of the model of Equation \eqref{eq:SIS} using the real function, $\beta_1$, in the generation model; (c) Absolute value of errors between $\beta_{\theta}$ and $\beta$; and (d) Change of loss function value over iterations.}
\label{fig:1}
\end{figure}

\begin{figure}[htbp]
\centering
\subcaptionbox{Hybrid model fitted result.}{\includegraphics[width=0.48\textwidth]{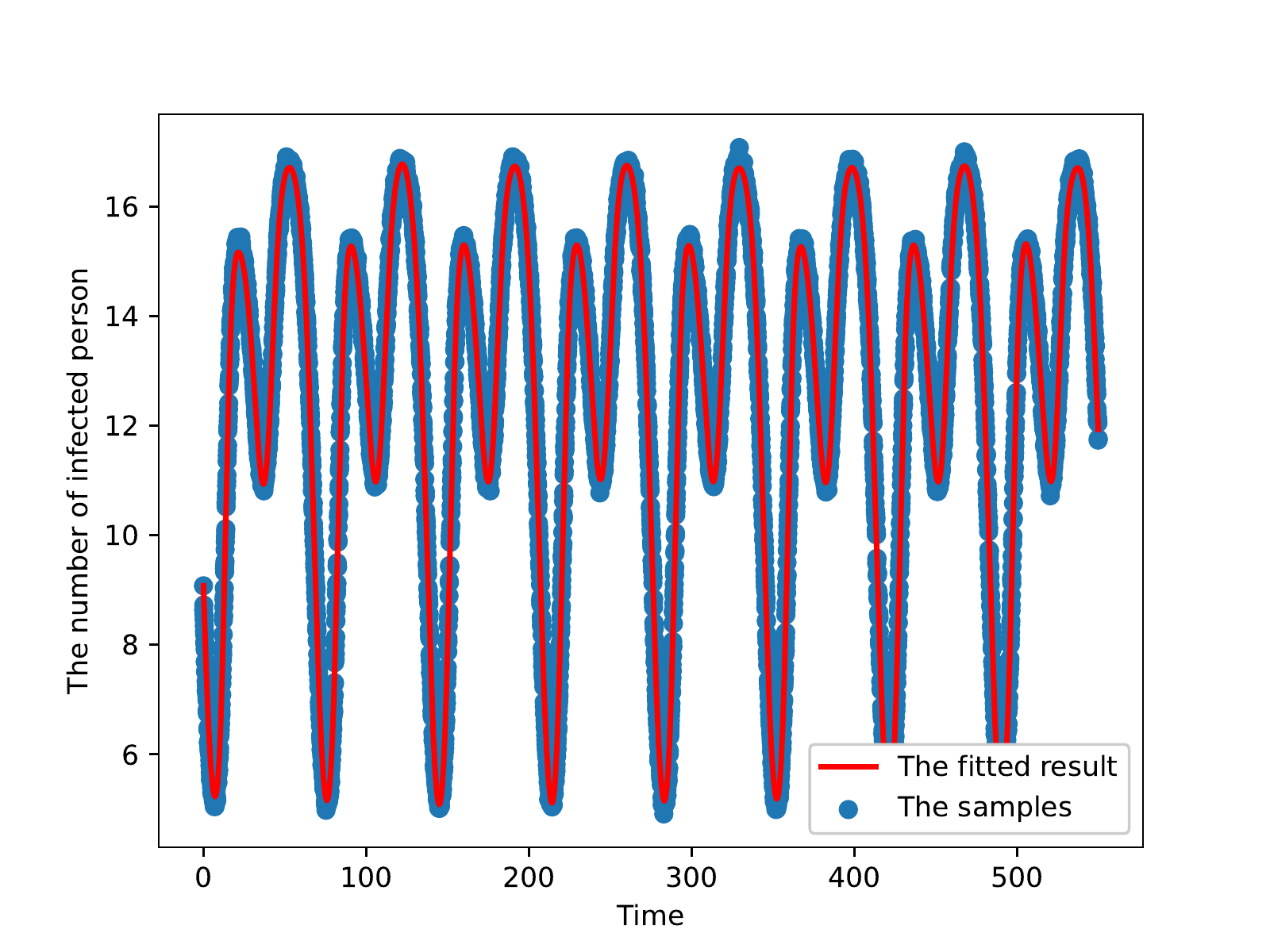}}%
\hfill
\subcaptionbox{Neural network fitted result.}{\includegraphics[width=0.48\textwidth]{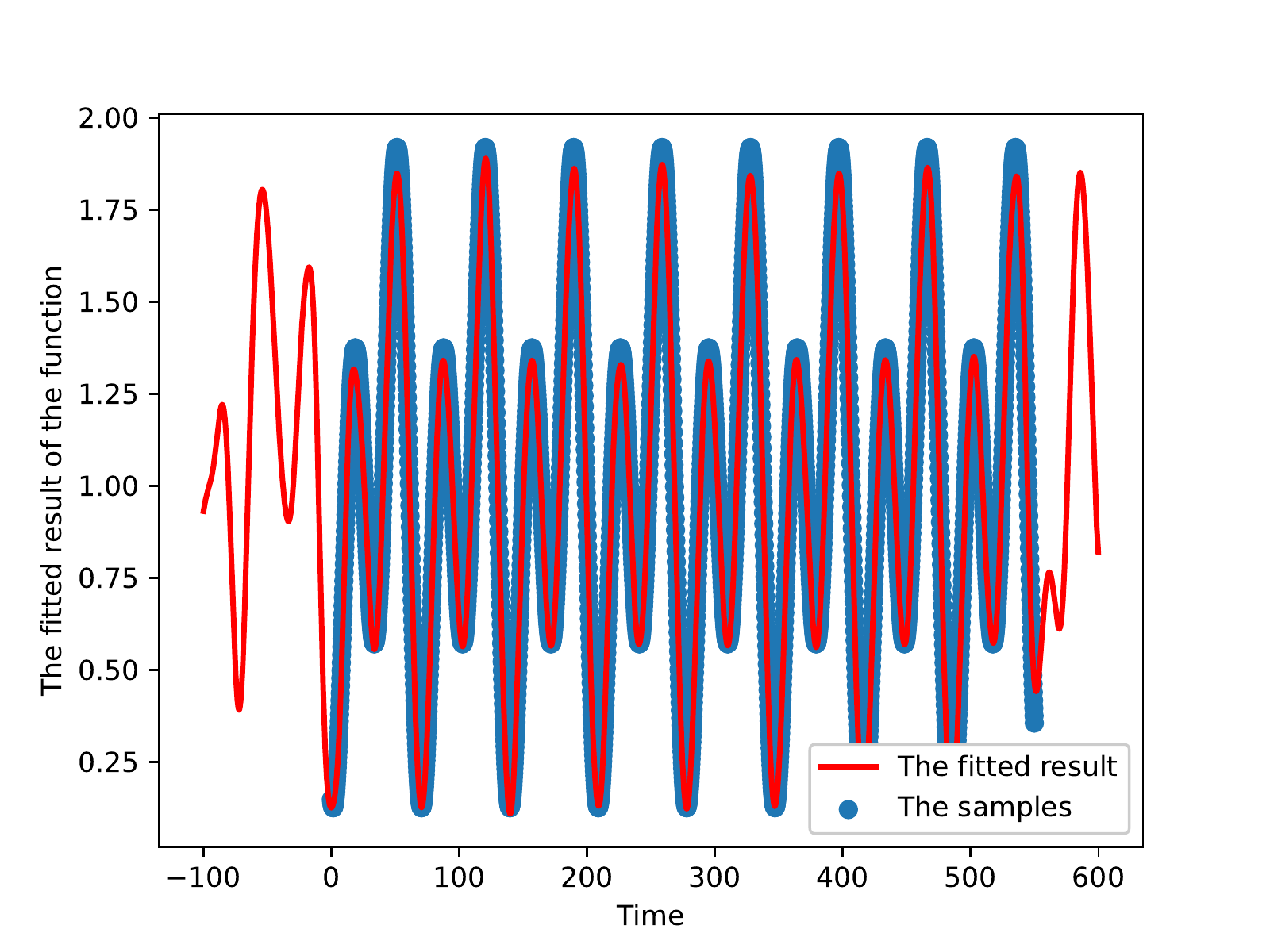}}%

\subcaptionbox{$|\beta(t)-\beta_{\theta}(t)|$.}{\includegraphics[width=0.48\textwidth]{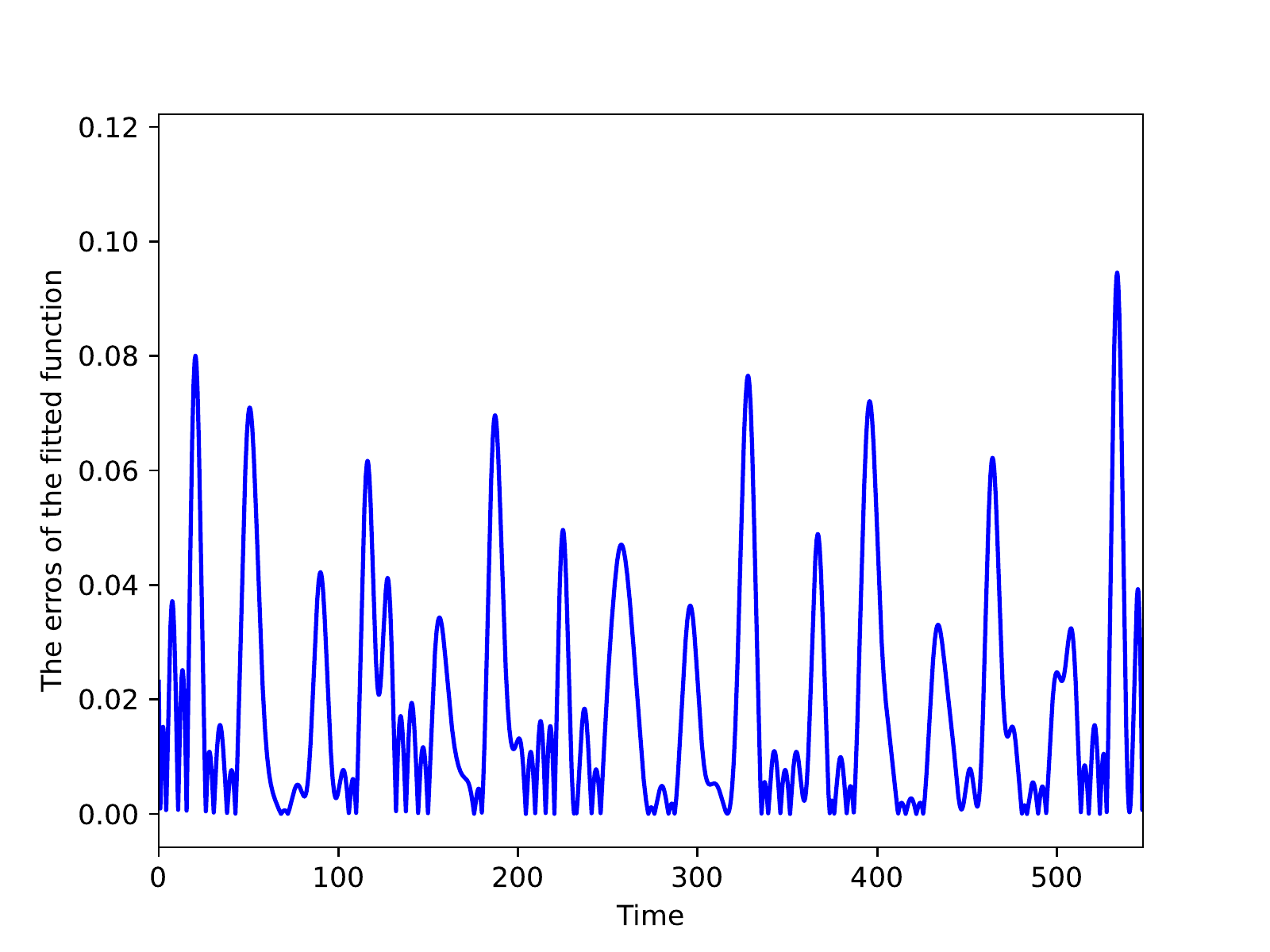}}%
\hfill
\subcaptionbox{Loss with respect to iterations.}{\includegraphics[width=0.48\textwidth]{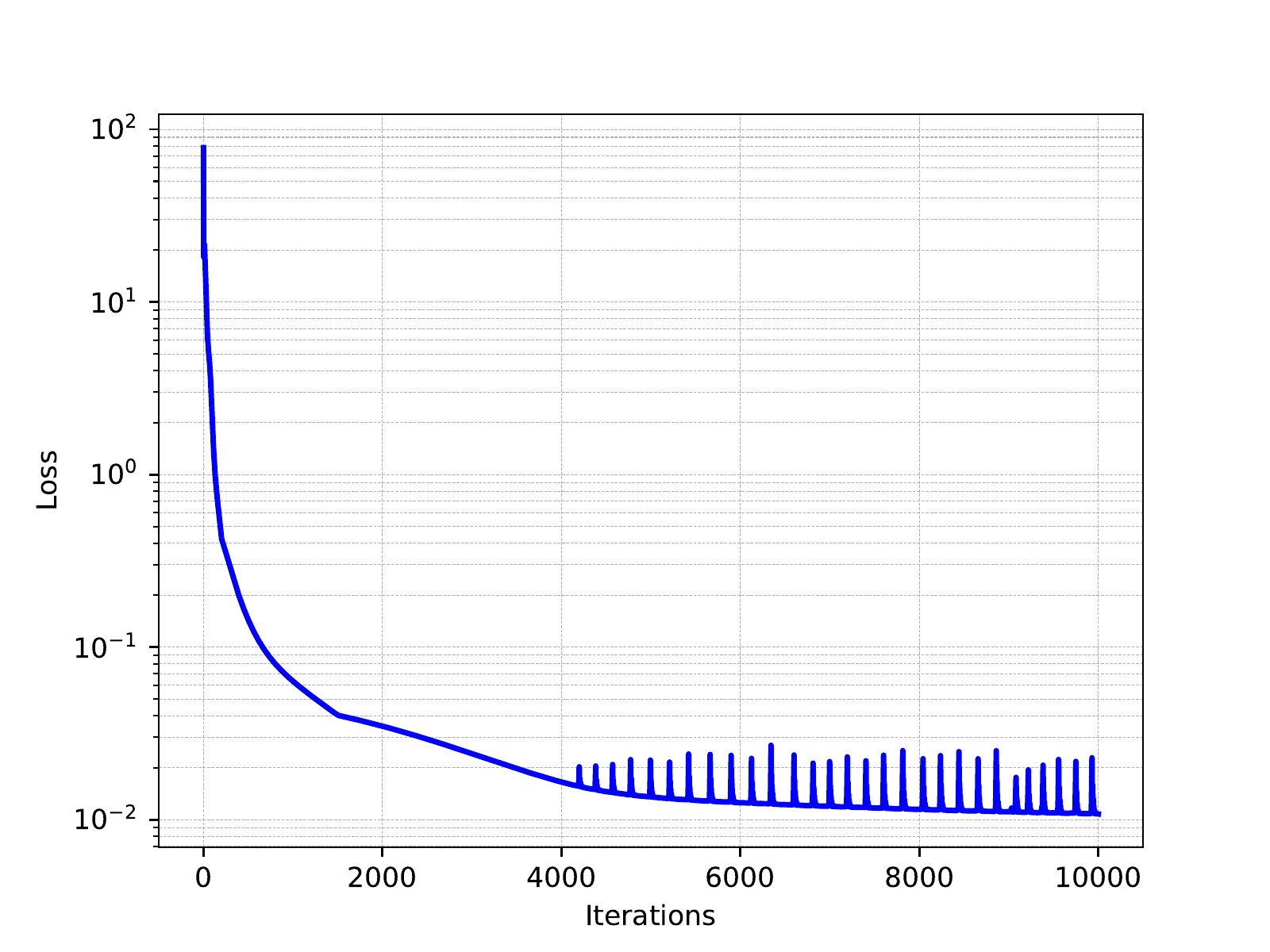}}%
\caption{Fitted results and training process of the susceptible--infected--susceptible model with complex periodic function: (a) Fitted result of the hybrid model of Equation \eqref{eq:SIS} with the generated data; (b) Fitted result of neural-network $\beta_{\theta}$ of the model of Equation \eqref{eq:SIS} with function $\beta_2$ in the generation model; (c) Absolute value of errors between $\beta_{\theta}$ and $\beta$; and (d) Change of loss function value over iterations.}
\label{fig:2}
\end{figure}

The first and second numerical experiments were based on the standard susceptible --infected--susceptible (SIS) model with birth and death data \cite{ma2009modeling}:
\begin{equation}
    \left\{
\begin{aligned}
    \frac{dS}{dt} &= \lambda -\beta(t) SI + \gamma I -dS, \\
    \frac{dI}{dt} &= \beta(t) SI - \gamma I - dI,
\end{aligned}\right.
\label{eq:SIS}
\end{equation}
where the $S$ is the number of susceptible individuals, and $I$ is the number of infected individuals at time $t$. Parameter $\lambda$ is the constant birth rate, $\gamma$ is the recovery rate, and $d$ is the death rate. Function $\beta(t)$ is the periodic function of time $t$ and is the infection function about the infection incidence. In the first experiment, we used the simply periodic function, $\beta_1(t) = 0.5\Big(\cos(t/11)+1\Big)$, as the infection function to generate data $D(t)$, and we used the hybrid model with the neural network as the $\beta(t)$ to fit these data. In the second experiment, the complex periodic function,

$$\beta_2(t) = 0.125\Big(-\cos(t/11)-\sin(t/11)-2.5\cos(2t/11)+0.5\sin(2t/11)\Big) + 0.5,$$ was used as the infection function. In the hybrid model of the second experiment, the neural network was used to fit the infection function. Equation \eqref{eq:SIS} is a special case of the system of Equation \eqref{eq:1}, with $S=x_S$ and $I=x_I$. In these two experiments, the neural-network models were the same. A three-layer neural network was used to equip the hybrid model with fully connected layers of $1\times 16$, $16\times 16$, and $16\times 1$ matrices and their related biases. Both activation functions in the first and second layers were the same: $\sin(ax)^2/a$ with $a=0.1$ \cite{ziyin2020neural}. The $R_0$ function used in the loss function of Equation \eqref{eq:m} was numerically written as $\frac{\sum_{i=1}^M\beta_{\theta}(i\Delta t)}{M(\gamma+d)}$ \cite{wang2008threshold}, where $\beta_{\theta}(t)$ is the fitted function of the neural network, and the constant, $M$, is the length of the time step. For hyperparameter $\alpha$ in Equation \eqref{eq:m}, we chose $300(\gamma+d)$.

Figures \ref{fig:1} and \ref{fig:2} illustrate the results of the first experiment regarding simple and complex periodic functions. During the training process of the first experiment, the initial learning rate was set to $1\times 10^{-2}$ and changed to $3.5\times 10^{-3}$ and $1.225\times 10^{-3}$ at the $250$th and $1,500$th iterations, respectively. During the training process of the second experiment, the initial learning rate was set to $1\times 10^{-2}$ and changed to $3\times 10^{-3}$ and $9\times 10^{-4}$ at the $200$th and $1,500$th iterations, respectively. In figures \ref{fig:1} and \ref{fig:2}, subfigures (a), (b), (c), and (d) present the respective fitted results of the hybrid model, the neural network, their errors, and the training process. From these two experiments, based on the loss function of Equation \eqref{eq:descrete}, we found that for the samples collected (blue points), the hybrid model fit the data well, and the neural network fit the infection function very well, which verifies the results of Theorem \ref{th:existence}. However, in the area lacking samples, the neural network did not fit the infection function well. Comparing the results of Figure \ref{fig:1} with Figure \ref{fig:2}, based on the same neural network model with the activation function, $\sin(ax)^2/a$, we can see that the complex periodic method produced better accuracy. During training, the simple periodic method had more vibrations. 

\begin{figure}[htbp]
  \centering
\subcaptionbox{Hybrid model fitted result.}{\includegraphics[width=0.48\textwidth]{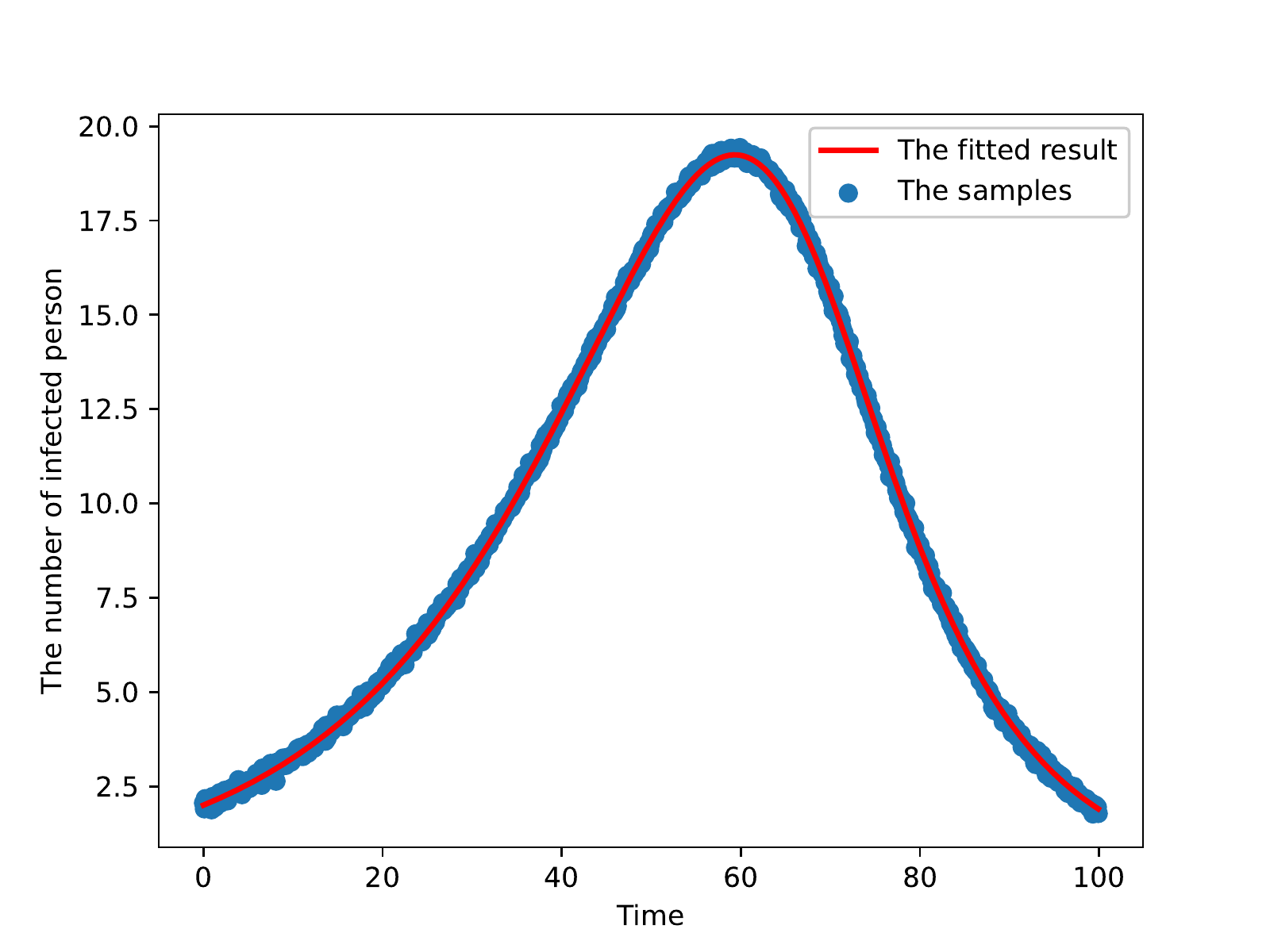}}%
\hfill
\subcaptionbox{Neural network fitted result.}{\includegraphics[width=0.48\textwidth]{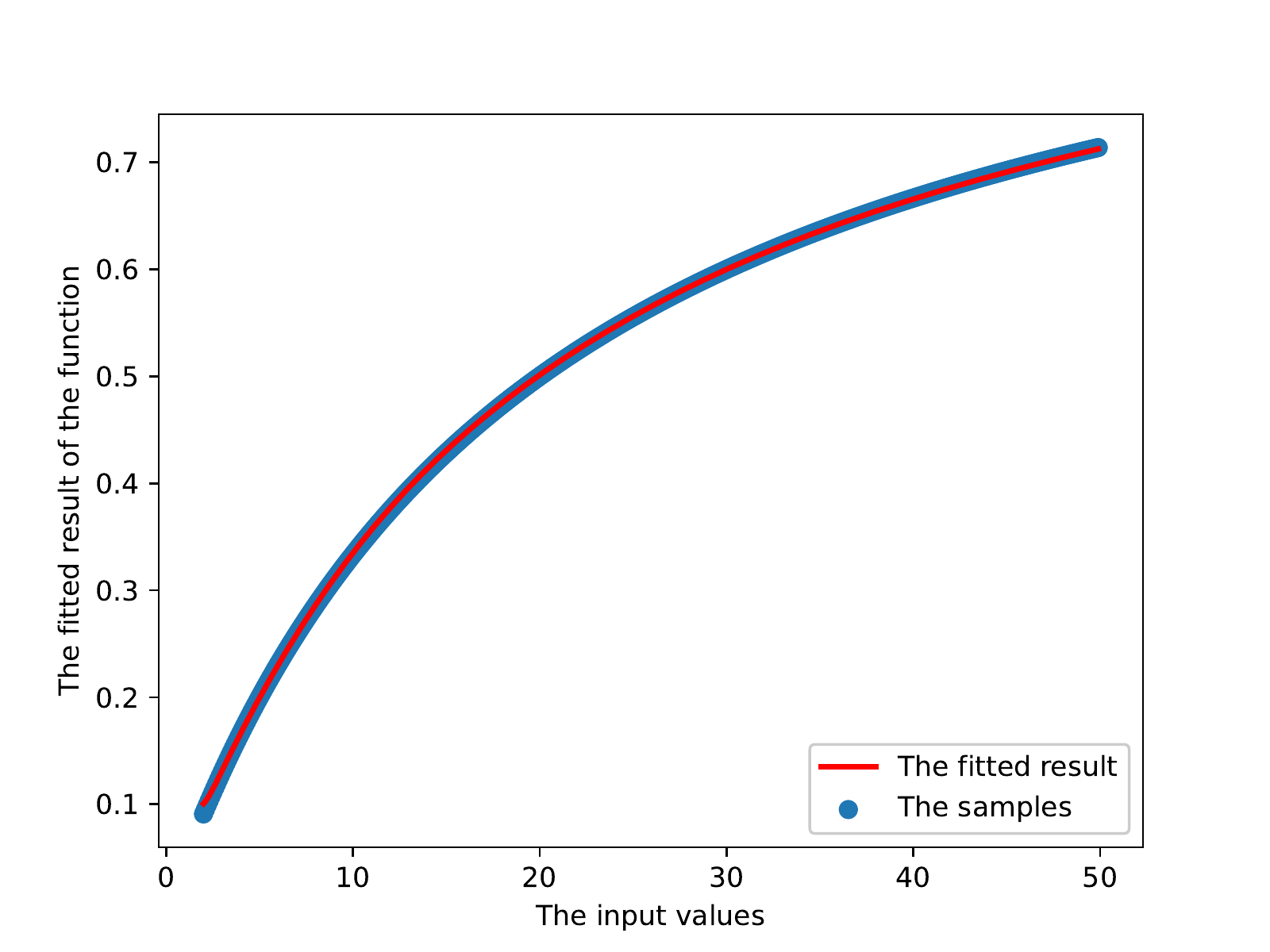}}%

\subcaptionbox{$|\beta(S)-\beta_{\theta}(S)|$.}{\includegraphics[width=0.48\textwidth]{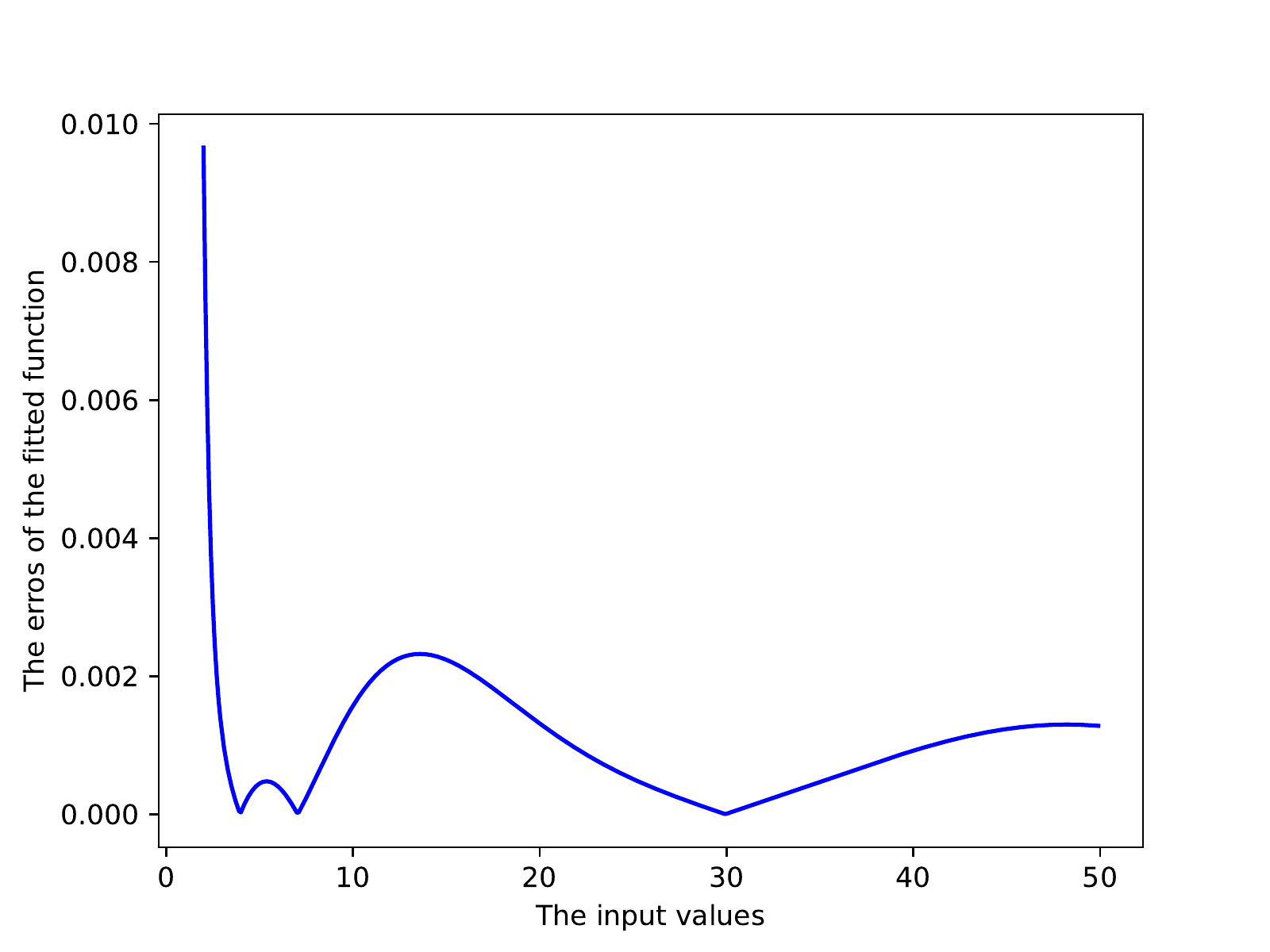}}%
\hfill
\subcaptionbox{Loss with respect to iterations.}{\includegraphics[width=0.48\textwidth]{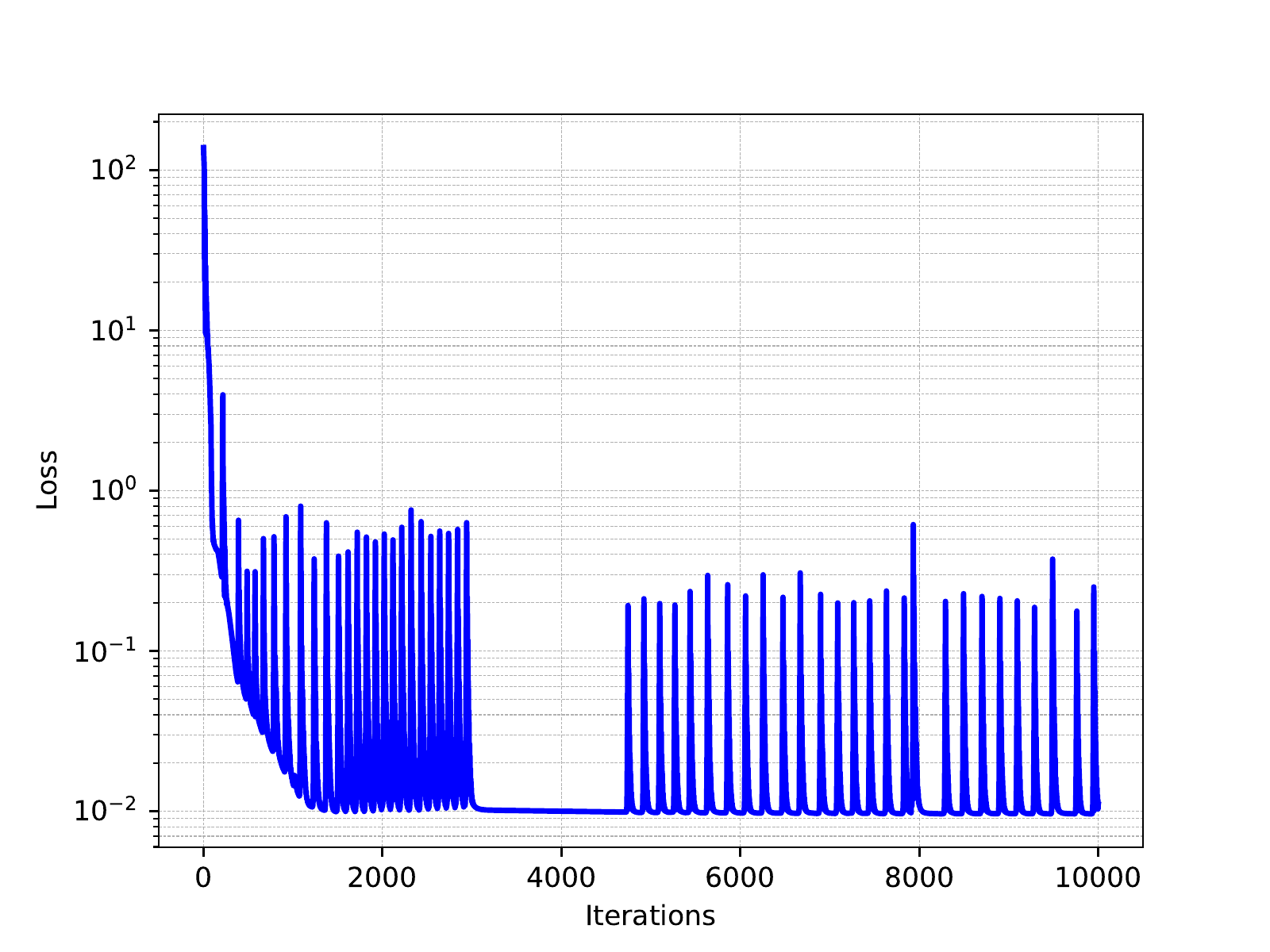}}%
\caption{Fitted results and training process of the susceptible--infected--recovered model with infection function $\beta(S)$: (a) Fitted result of the hybrid model of Equation \eqref{eq:SIRS} with the generated data; (b) Fitted result of neural-network $\beta_{\theta}$ of model of Equation \eqref{eq:SIRS} with function $\beta(S)$ in the generation model; (c) Absolute value of errors between $\beta_{\theta}$ and $\beta$; and (d) Change of loss function value over iterations.}
  \label{fig:3}
\end{figure}

\begin{figure}[htbp]
  \centering
\subcaptionbox{Hybrid model fitted result.}{\includegraphics[width=0.48\textwidth]{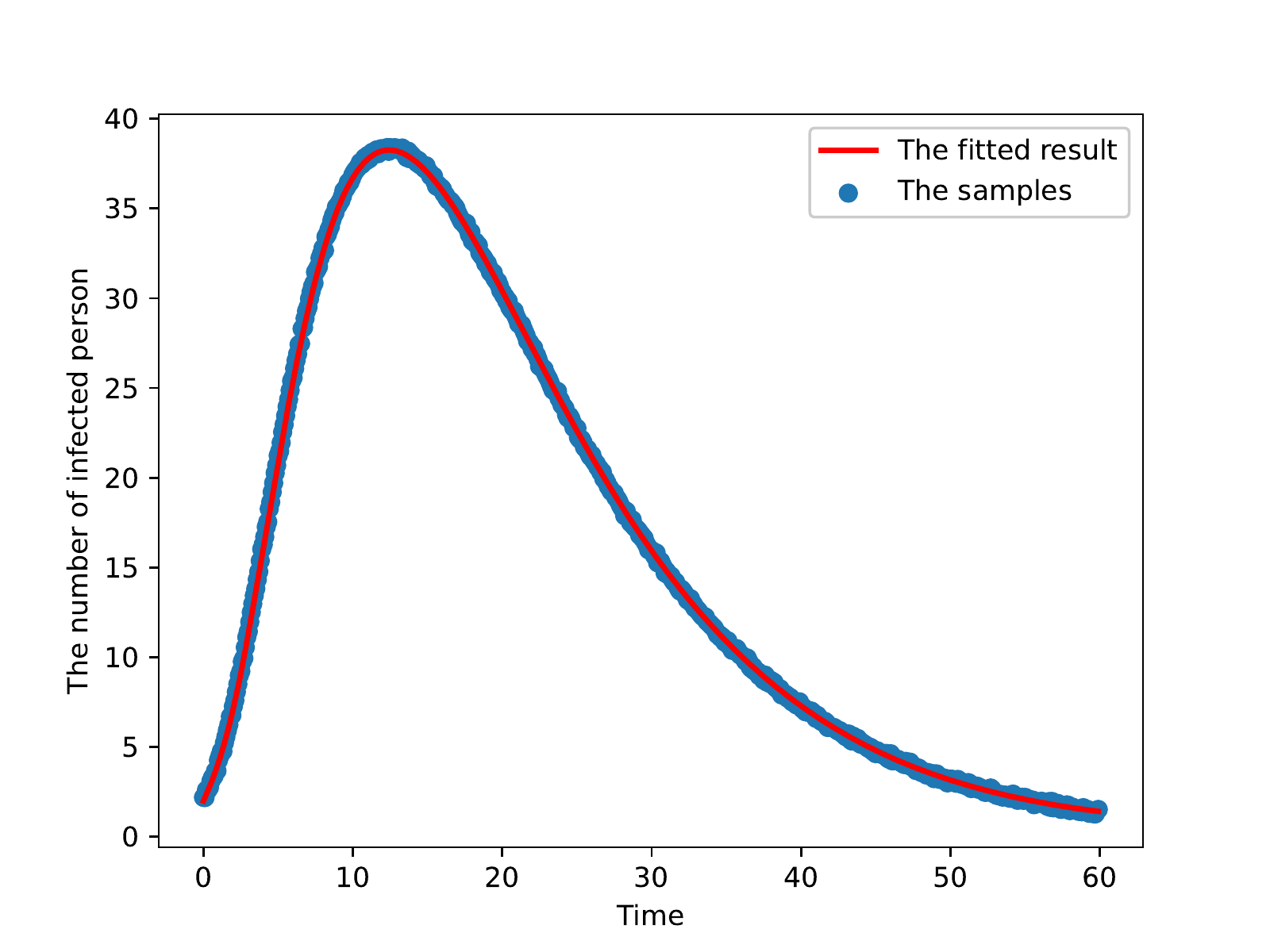}}%
\hfill
\subcaptionbox{Neural network fitted result.}{\includegraphics[width=0.48\textwidth]{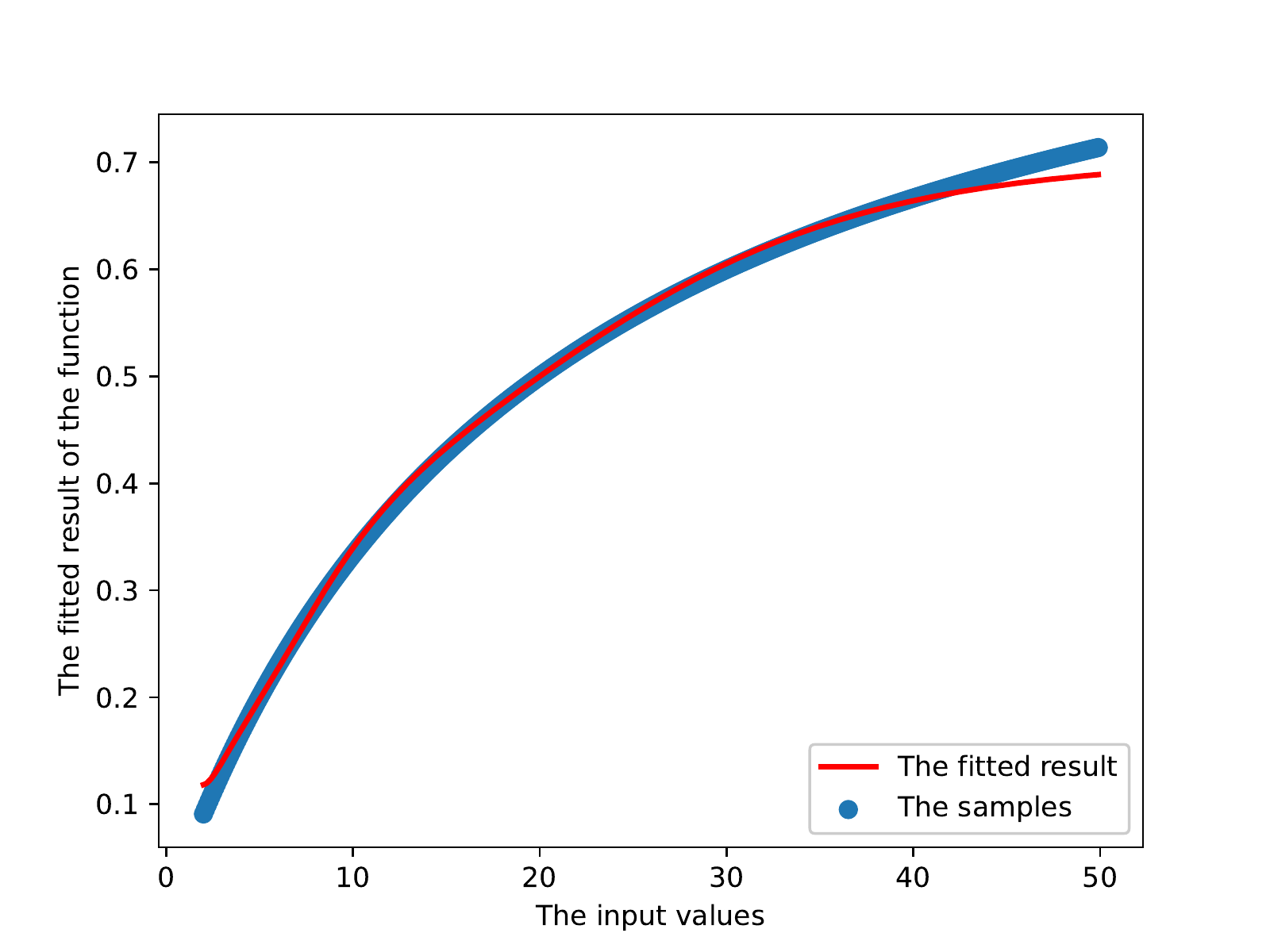}}%

\subcaptionbox{$|\beta(I)-\beta_{\theta}(I)|$.}{\includegraphics[width=0.48\textwidth]{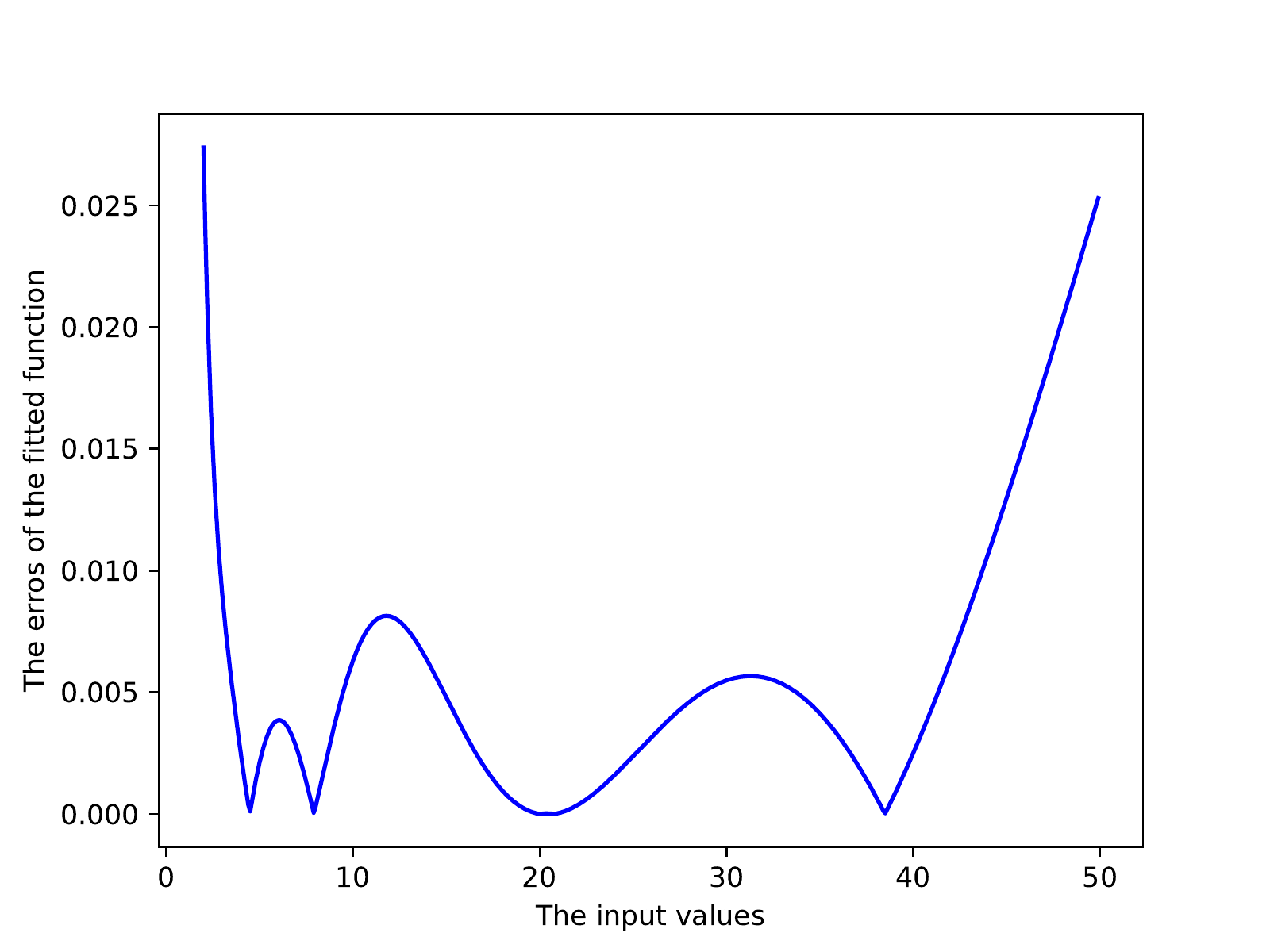}}%
\hfill
\subcaptionbox{Loss with respect to iterations.}{\includegraphics[width=0.48\textwidth]{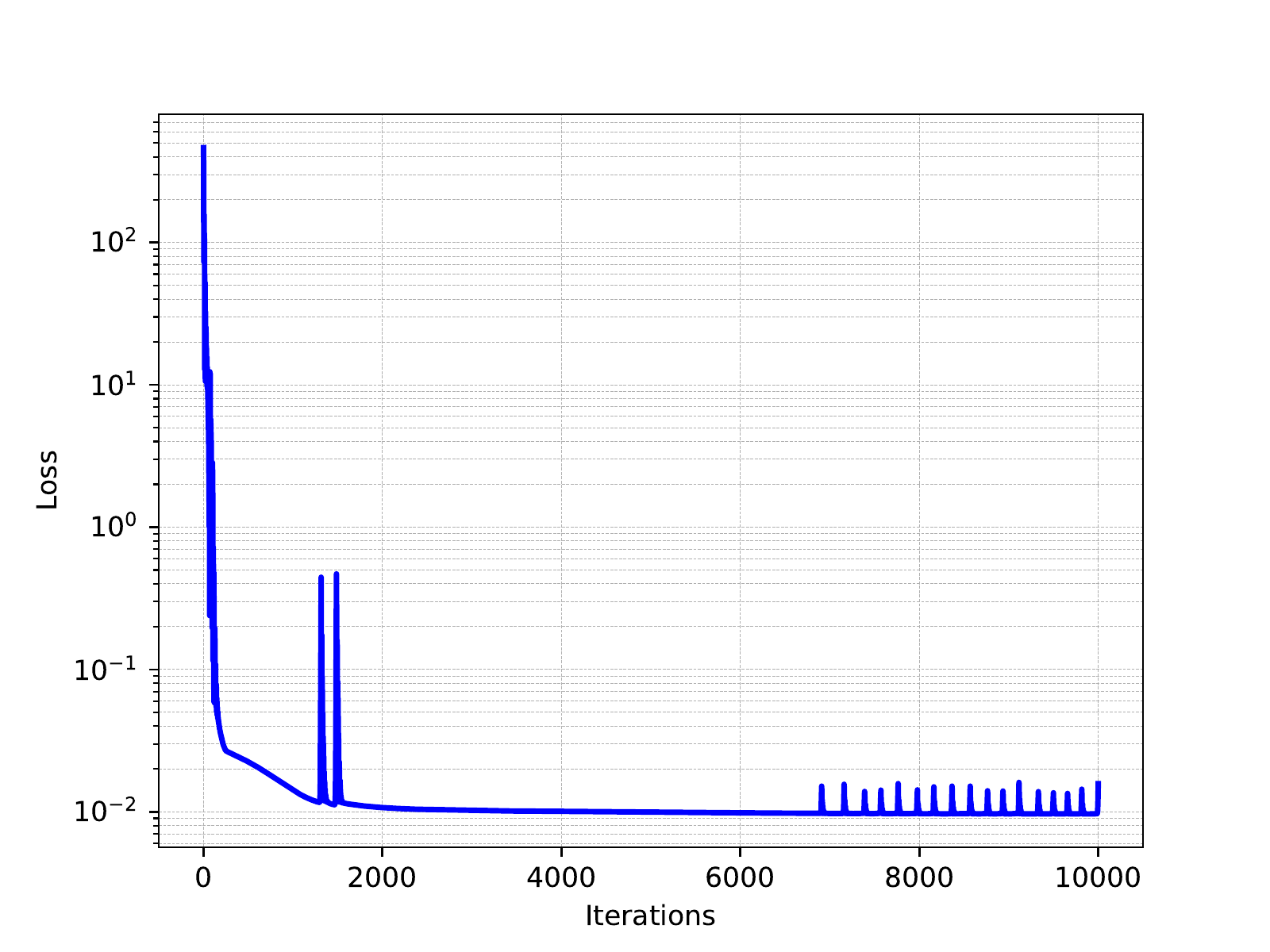}}%
\caption{Fitted results and training process of the susceptible--infected--recovered model with infection function $\beta(I)$: (a) Fitted result of the hybrid model of Equation \eqref{eq:SIRI} with the generated data; (b) Fitted result of neural-network $\beta_{\theta}$ of the model of Equation \eqref{eq:SIRI} with function $\beta(I)$ in the generation model; (c) Absolute value of errors between $\beta_{\theta}$ and $\beta$; and (d) Change of loss function value over iterations.}
  \label{fig:4}
\end{figure}

The third and fourth experiments were based on the susceptible--infected-- recovered  (SIR) model with different infection functions correlated with variables $S$ and $I$, which can be written as $\beta(S)$ and $\beta(I)$ in this case.
The model about the $\beta(S)$ is listed as
\begin{equation}
    \left\{
\begin{aligned}
    \frac{dS}{dt} &= -\beta(S)I, \\
    \frac{dI}{dt} &= \beta(S)I - \gamma I, \\
    \frac{dR}{dt} &=  \gamma I,
\end{aligned}\right.
\label{eq:SIRS}
\end{equation}
where $S$ and $I$ are same as those in the SIS model, and variable $R$ is the number of recovered individuals at time $t$. Parameter $\gamma$ is the recovery rate. This model is used to generate data and construct the hybrid model for the third experiment. Similarly, the model about $\beta(I)$ is listed as
\begin{equation}
    \left\{
\begin{aligned}
    \frac{dS}{dt} &= -\beta(I)S, \\
    \frac{dI}{dt} &= \beta(I)S - \gamma I, \\
    \frac{dR}{dt} &=  \gamma I.
\end{aligned}\right.
\label{eq:SIRI}
\end{equation}
This model is used in the fourth experiment. Infection incidences $\beta(S) = 0.2S/(S+20)$ and $\beta(I) = 0.2I/(I+20)$ appear in the third and fourth experiments, respectively. These two SIR models are equivalent to the previous ones and represent special cases of the system of Equation \eqref{eq:1} with $I=x_I$ and $S=x_S$. In both experiments, we used a neural-network model with three fully connected layers to fit the infection functions, and the activation functions in the first and second layers were the same: $\tanh(x)$. The three layers were constructed with $1\times 16$, $16\times 16$, and $16\times 1$ matrices and their related biases. In the loss function of Equation \eqref{eq:descrete}, the $R_0$ for the $\beta(S)$ system was $\frac{\beta_{\theta}(S_0)}{\gamma}$, and for the $\beta(I)$ system, it was $\frac{10(\beta_{\theta}(0.1)-\beta_{\theta}(0))S_0}{\gamma}$, where $S_0$ was the initial condition of the hybrid model, and function $\beta_{\theta}$ was the neural-network model. The hyperparameter, $\alpha$, was set to $300\gamma$.

Figures \ref{fig:3} and \ref{fig:4} respectively illustrate the results of the third experiment about the infection function, $\beta(S)$, and the fourth experiment about the infection function, $\beta(I)$. During the training process of the third experiment, the initial learning rate was set to $1\times 10^{-2}$ and changed to $4\times 10^{-3}$ and $1.6\times 10^{-3}$ at the $250$th and $3,000$th iterations, respectively. During the training process of the fourth experiment, the initial learning rate was set to $4\times 10^{-2}$ and changed to $1\times 10^{-2}$, $2.5\times 10^{-3}$, and $6.25\times 10^{-4}$ at the $250$th, $1,500$th, and $3,000$th iterations, respectively. In Figures \ref{fig:3} and \ref{fig:4},  subfigures (a), (b), (c), and (d) present the respective fitted results of the hybrid model, the neural network, their errors, and the training process, respectively. From the results of these two experiments, we can conclude that the hybrid model fit the data well, and the neural network model fit the infection function well. Furthermore, the distribution of the errors illustrate no obvious laws. Comparing Figure \ref{fig:3} with Figure \ref{fig:4}, during the training process, the process of estimating $\beta(S)$ had more vibrations. 

\section{Application to real data}

In this section, we apply the proposed methods to predict the ground-truth changes in historical COVID-19 infectivity. U.S. COVID-19 data were used, collected from the official World Health Organization (WHO, \url{https://covid.cdc.gov/covid-data-tracker/#trends_dailycases}) website. Time periods, June 16 to October 31, 2021, and December 12, 2021, to February 11, 2022, were used for model fitting. These periods present the two most recent case peaks in 2021--2022. During COVID-19 transmission, the dominant strains were delta and omicron, respective to the two peaks \cite{mohapatra2022twin}. To reduce the impact of decreased case detection on weekends and the subsequent increases on Mondays, a 7-day moving average was used. The data are available from the WHO website.

\begin{figure}[htbp]
  \centering
\subcaptionbox{Fitted result of the hybrid model of Equation \eqref{eq:6.2} during the first period.}{\includegraphics[width=0.48\textwidth]{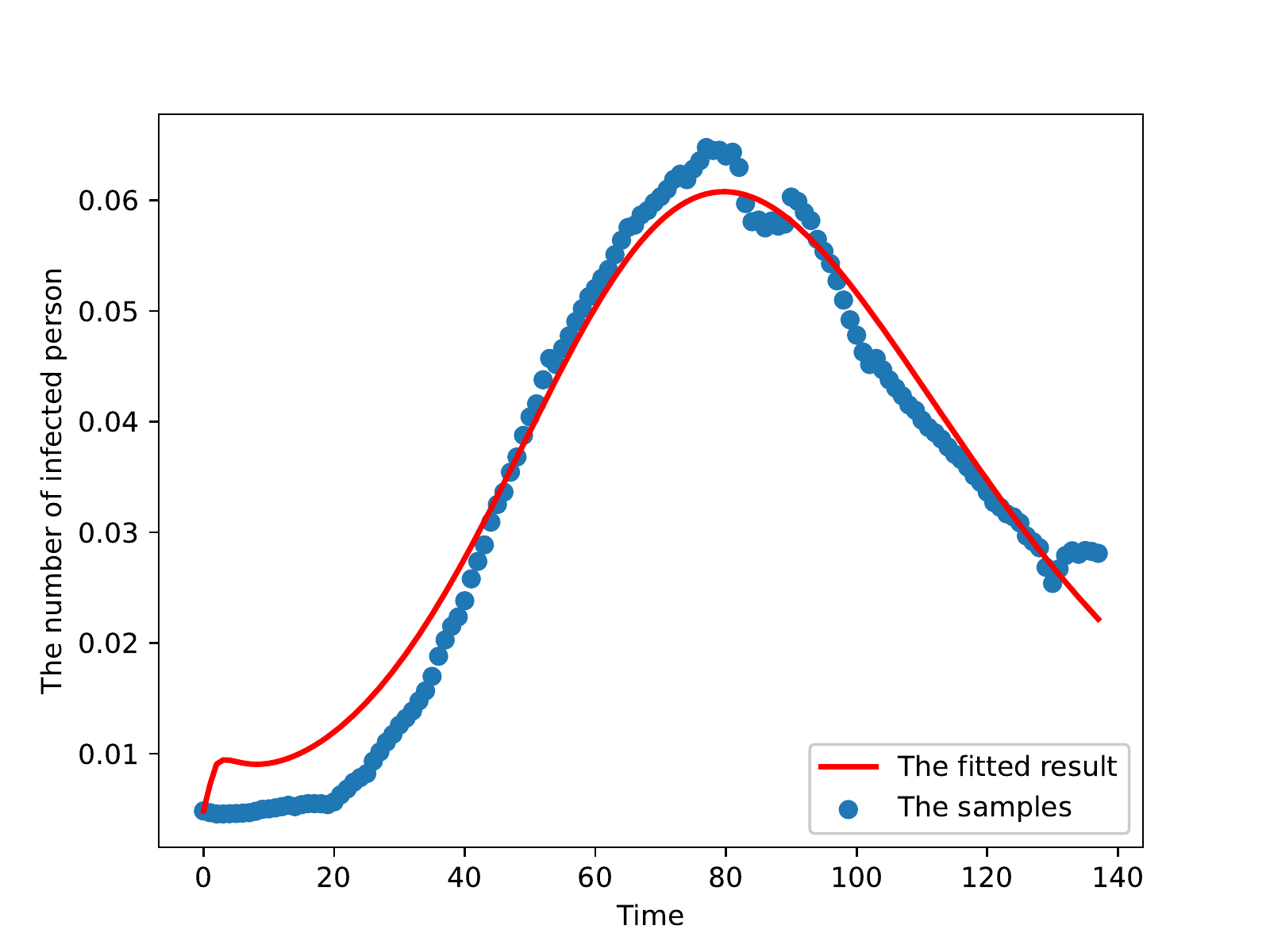}}%
\hfill
\subcaptionbox{Fitted result of the hybrid model of Equation \eqref{eq:6.2} during the second period.}{\includegraphics[width=0.48\textwidth]{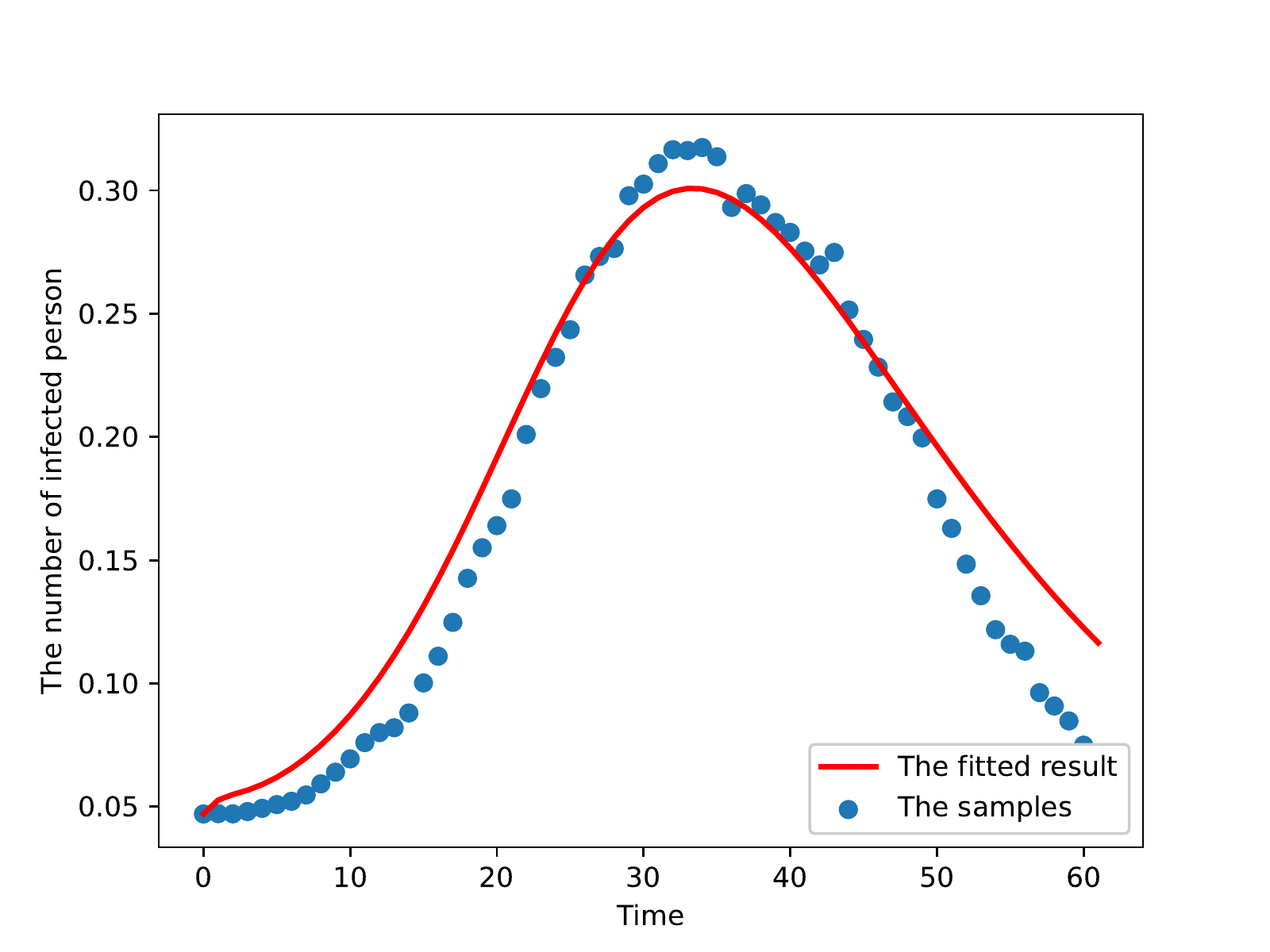}}%

\caption{Fitted results of the model of Equation \eqref{eq:6.2} with two datasets collected from (a) June 16 to October 31, 2021, and (b) December 12, 2021, to February 11, 2022, from U.S. COVID-19 case data. }
  \label{fig:7}
\end{figure}

To appropriately fit the data so that the infection function properties can be studied, we introduce a modified SIR model:
\begin{equation*}
\left\{
\begin{aligned}
    \frac{d\hat{S}}{dt} &=  -\beta\hat{S}I,\\
    \frac{dI}{dt} &= \beta\hat{S}I -(\gamma +d)I,
\end{aligned}\right.
\end{equation*}
where $\hat{S}=S\times 10^{-3}$; parameter $\gamma$ is the recovery rate, and $d$ is the death rate. In the model, the $\hat{S}$ is a variable of the infection function, $\beta\hat{S}I$, which shares its trend with $\beta SI$ of the standard SIR model. We also introduce a sufficiently large constant, $M$, such that for all times $t$, $\beta_M=M\beta$, $s=\hat{S}/M$, and $i=I/M$ follow $s(t)<1$ and $i(t)<1$ to produce the following model equation:
\begin{equation}
\left\{
\begin{aligned}
    \frac{ds}{dt} &=  -\beta_M si,\\
    \frac{di}{dt} &= \beta_M si -(\gamma +d)i.
\end{aligned}\right.
\label{eq:6.2}
\end{equation}
The initial values were $s_0=1-10^{-3}i_0$, $i_0=I_0/M$, and $i(t)=I(t)/M$. For calculation, we used $M=7.6N\times 10^{-3}$, where $N$ is the population of the U.S. During transmission, temperature changes and behavioral changes affected the infection rates of the two key virus strains. Infection rate $\beta$ is a time-related function, $\beta(t)$. Thus, in the above model of Equation \eqref{eq:6.2}, we can assume that $\beta=\beta(t)$ is the function that must be fitted by the neural network.

\begin{figure}[htbp]
\centering
  \includegraphics[width=0.5\textwidth]{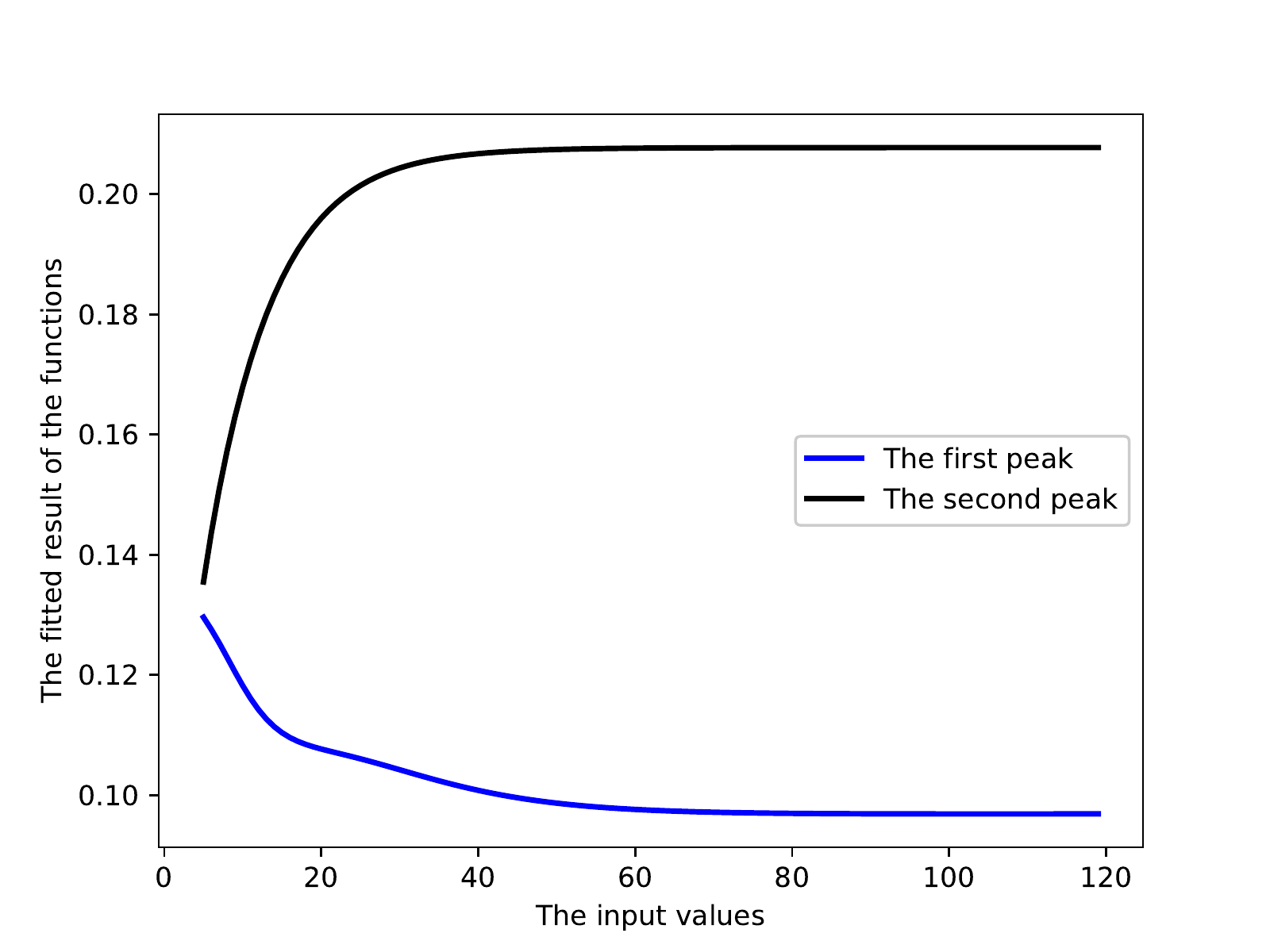}
  \caption{Fitted result of infection function $\beta(t)$ of the model of Equation \eqref{eq:6.2} using data collected from two time periods around the two peaks of the COVID-19 case data.}
\label{fig:8}
\end{figure}

Figures \ref{fig:7} (a) and (b) illustrate the model-fitted results of first and second peak U.S. COVID-19 case data, and Figure \ref{fig:8} shows the fitted results of the infection function of these two peaks. In the first period, the infectivity first decreased and then became constant. In the second period, the infectivity first increased and then became constant. However, after comparing the fitted results of the two infection functions, we can conclude that the change rate of the infection function in the second period was faster than that of the first. Similarly, the second period had a higher infection rate than the first. The decrease in the first period may have been caused by the increased number of vaccinated people, and the increase in the second may have been caused by the introduction omicron strain and the decrease in temperature. 

Through this practical application example, it was demonstrated that our method can interpret otherwise hidden information in limited real-world data, which not only increases application breadth but also greatly increases model adaptability. The proposed method establishes a theoretical basis for further research on infectious diseases and provides algorithmic support for the same.

\section{Conclusion and discussion}

In this work, by introducing the basic properties of an epidemic model and the forward bifurcation of an epidemic system, we explained the vanishing gradient conditions that result in the untrainability of the hybrid model. A novel loss function that combines MSE loss and bifurcation items was then constructed to overcome this issue. The existence conditions needed to ensure the estimability of the neural network were also discovered. Based on the standard fixed-step numerical method, we trained our hybrid model with generated data, and the numerical results verify its accurate estimability. To our knowledge, this is the first work to theoretically analyze and empirically demonstrate a trainable nonlinear hybrid model that accurately predicts the behavior of real-world behaviors. 

In Section 5, the results of four numerical experiments were described to demonstrate the applicability and efficacy of the hybrid model in estimating periodical and Holling type infection functions. All experiments show that the system of Equation \eqref{eq:1} can fit the data very well, validating Theorem \ref{th:existence}. We also provided the necessary parameters to complete the experiments alongside a specific implementation method and its published programming code, all of which may be used in related hybrid-model applications. In Section 6, we reported the application of our method using real COVID-19 data from the U.S., and the results showed that the proposed model indeed finds the hidden information behind limited real data, which increases its applicability to infectious disease models.

This work provides exciting new sights into the capabilities of hybrid neural networks in handling nonlinear problems via nonlinear ordinary differential equations. In a future work, we plan to focus on more complex but more widely used stochastic and partial differential equations. Such follow-on studies will extend the range of the application of hybrid models while enabling them to solve complex econometric, computational fluid mechanic, and computational chemistry problems. Furthermore, deep-learning models’ great dependency on manually labeled training data may be mitigated. As the hybrid model is based on many prior conditions, it can rely on smaller neural networks, which reduces the need for training data.

\bibliographystyle{plain}
\bibliography{references}
\end{document}